%% file: CMCLiouvilleMain.tex
\documentclass[12pt,reqno]{amsart}
\usepackage{hyperref}
\usepackage{upgreek}
\usepackage[mathscr]{euscript}

\usepackage{libertine}

\usepackage{enumitem}
\usepackage{xcolor} 
\usepackage{tikz}
\usetikzlibrary{shapes,arrows}

\newcommand\fsu{\mathfrak{su}}

\DeclareRobustCommand{\rchi}{{\mathpalette\irchi\relax}}
\newcommand{\irchi}[2]{\raisebox{\depth}{$#1\chi$}}

\newcommand\defem{\sl}

\newcommand\wbar{\overline{w}}
\newcommand\pbar{\overline{p}}

\newcommand\qbar{\overline{q}}

\def\Winf{W^{(\infty)}}
\def\Rinf{R^{(\infty)}}
\def\np{r}
\def\CCR{R}
\newcommand\scrU{\mathscr U}
\newcommand\bbar{\overline{b}}
\newcommand{\pibars}[1]{\overline{\pi^{#1}}}
\newcommand\sysJU{{\mathcal J}\vert_{\mathscr U}}

\newcommand\muG{\mu}

\newcommand\bq{\Pi}
\def\prW{\widehat{W}}
\def\prV{\widehat{\widetilde{V}}}
\def\prR{\widehat{R}}
\def\prRinf{\widehat{R}^{(\infty)}}

\def\rhoSO{\rho}
\def\rhoSU{\rho_{SU(2)}}

\newcommand\sfj{\mathrm j}

\newcommand\di{\partial}

\newcommand\momega{\upomega}

\newcommand\momegabar{\overline{\momega}}

\newcommand\Borel{\mathbb B}
\newcommand\R{\mathbb R}
\newcommand\CP{\mathbb C\mathrm{P}}
\newcommand\ri{\mathrm{i}}
\newcommand\re{\mathrm{e}}
\newcommand\meta{\upeta}
\newcommand\metabar{\overline{\upeta}}

\newcommand\F{{\mathscr F}}
\newcommand\Lie{{\mathcal L}}

\newcommand\calD{\mathcal D}

\newcommand\sysI{\mathcal I}
\newcommand\sysICMC{\sysI_{CMC}}
\newcommand\tilsysICMC{\widetilde{\sysI}_{CMC}}
\newcommand\sysMCMCpr{pr(\mathcal M_{CMC})}

\newcommand\rd{\operatorname{d}\!} 

\newcommand\sysJ{\mathcal J}
\newcommand\sysBpr{\mathcal J}

\newcommand\sysE{\mathcal E}
\newcommand\sysH{\mathcal H}
\newcommand\sysM{\mathcal M}
\newcommand\sysB{\mathcal B}
\newcommand\sysMCMC{\mathcal M_{CMC}}

\newcommand\tilsysMCMC{\widetilde{\mathcal M}_{CMC}}

\newcommand\tilV{\widetilde{V}}

\newcommand\muS{\mu}

\newcommand\sysEpr{pr(\mathcal E)}

\newcommand\diff{\text{diff}}
\newcommand\alg{\text{alg}}
\renewcommand\span{\operatorname{span}}
\newcommand\spanC{\operatorname{span}_{\mathbb C}}
\newcommand\spanR{\operatorname{span}_{\mathbb R}}

\newcommand\HT{{\mathrm H^3 }}
\newcommand\J{{\mathrm J}}

\newcommand\vecC{{\mathsf X}}  
\newcommand\vecCY{{\mathsf Y}}
\newcommand\hatvecCY{\widehat{\mathsf Y}}
\newcommand\Ofour{{\mathcal O(4)}}   

\newcommand\Cq{{\pi}} 

\newcommand\qpt{{\mathrm q}} 
\newcommand\upt{{\mathrm u}}

\newcommand\realversion{{\rm real version \ }}
\renewcommand{\Re}{\operatorname{Re}}
\renewcommand{\Im}{\operatorname{Im}}

\newcommand{\ve}{\mathsf{e}}

\newcommand{\uve}[1]{\underline{\ve_{#1}}}

\newtheorem{thm}{Theorem}[section]
\newtheorem{lem}[thm]{Lemma}
\newtheorem{corollary}[thm]{Corollary}
\newtheorem{remark}{Remark}

\newcommand\reals{{\rm I \! R }}  

\newcommand\C{{\mathbb C}}

\DeclareFontFamily{U}{mathx}{}
\DeclareFontShape{U}{mathx}{m}{n}{<-> mathx10}{}
\DeclareSymbolFont{mathx}{U}{mathx}{m}{n}
\DeclareMathAccent{\widehat}{0}{mathx}{"70}
\DeclareMathAccent{\widecheck}{0}{mathx}{"71}
\newcommand{\Cm}{{\C^{\times}}}

\begin{document}
\title[The Darboux Integrability of the Constant Mean Curvature One Equation]{On the Darboux Integrability of the Constant Mean Curvature One Equation for Surfaces in Hyperbolic 3 space}

\author{Mark E. Fels}
\address{Department of Mathematics and Statistics, Utah State University}
\email{mark.fels@usu.edu}
\author{Thomas A. Ivey}
\address{Dept. of Mathematics, College of Charleston}
\email{iveyt@cofc.edu}
\date{\today}

\begin{abstract}  We show, using group theoretic constructions, that the system of elliptic partial differential equations whose solutions determine the constant mean curvature immersions into hyperbolic 3 space are equivalent to the elliptic Liouville equation.  In particular, the Darboux integrability of each equation shows
that each of these equations admit an equivalent quotient representation which descends to an equivalence of the equations.  The explicit map providing the equivalence is given.
\end{abstract}
\maketitle

\input introCMC1.tex

\input section3.tex

\input section4.tex

\input appendix.tex

\input biblio.tex

\end{document}

%% file: introCMC1.tex
\section{Introduction}


One form of the elliptic Liouville equation is the PDE
\begin{equation}
U_{xx} + U_{yy} + 2  {\rm e}^U =0
\label{LiouEq}
\end{equation}
whose  general solution can be expressed in terms of an arbitrary holomorphic function $f(z)$ by
\begin{equation}
U =2 \log \frac {2|f'(z)|}{1+|f(z)|^2} .
\label{Liousol}
\end{equation}
The general solution to the PDE for surfaces of constant mean curvature one (CMC-1) in 
3-dimensional hyperbolic space 
$\HT$ also admits a closed form solution in terms a of holomorphic function; see, e.g., Theorem 1.1 in  \cite{LeviCMC}. The  fact that the general solution to the Liouville and the CMC-1 equations can be can be expressed in terms of holomorphic functions is directly related to the fact that both of these equations are Darboux integrable. Specifically, as shown in \cite{FelsIvey}, every Darboux integrable elliptic system admits what we call a {\it quotient representation} and this in this turn leads to the property that the solutions or integral manifolds of a Darboux integral elliptic system can be expressed in terms of holomorphic data. In this note we will show, using the quotient representation of these two equations, that they are equivalent (see Theorem \ref{LCMC}). This demonstrates how the quotient representation of Darboux integrable systems can be used not only to find solutions to differential equations in terms of holomorphic data but to also determine non-trivial equivalences and are fundamental in their geometric properties. Other applications of the quotient representation in the case of hyperbolic Darboux integrable systems can be found in \cite{AFB}, \cite{AFCP}; a detailed comparison of elliptic and hyperbolic systems and their quotient representations is given in \cite{FelsIveyArchive}.

In order to describe the quotient representation of a Darboux integrable elliptic system,  we must first define the quotient of an exterior differential system $\sysE$ on a manifold $N$ by a Lie group $G$ acting on $N$ by symmetries of
$\sysE$. Assume $\pi:N \to N/G$ is a smooth submersion of manifolds where $N/G$ is quotient space of orbits of $G$, The quotient $\sysE/G$ is the exterior differential system
\begin{equation}
\sysE/G = \{ \theta \in \Omega^*(N/G) \ | \  \pi^* \theta \in \sysE \ \}.
\label{Qsys}
\end{equation}
See \cite{AF1}, \cite{AFB}, \cite{AFV} for some properties and examples of quotient systems.

Suppose now that $(M,{\mathcal I})$ defines an elliptic Darboux integrable system.  Then the main result in \cite{FelsIvey} asserts that there exists a complex manifold $N$ and a holomorphic Pfaffian system ${\mathcal H}$ 
on $N$ that admits a complex Lie group of symmetries $K$ acting freely on $N$, and real form $G \subset K$ such that
$$
(M,{\mathcal I}) = (N/G, \sysE /G).
$$
where 
$\sysE$ is the \realversion of ${\mathcal H}$, i.e., the real Pfaffian which is generated by the real and imaginary parts of 1-forms in ${\mathcal H}$ and which has real rank twice the complex rank of ${\mathcal H}$.  
We say $(N,{\mathcal H}, K,G)$ defines the quotient representation of $(M,{\mathcal I}) $.   The quotient mapping 
$\pi:N \to N/G \equiv M$ takes holomorphic integral manifolds of ${\mathcal H}$ surjectively (locally) to integral manifolds in ${\mathcal I}$. The quotient representation leads to the closed form solution \eqref{Liousol} to the Liouville equation as well as the closed form solution to CMC-1 given in \eqref{IntroCMC} below (see also \cite{LeviCMC}).  
The equivalence of the Liouville equation and constant mean curvature one equation for surfaces in $\HT$ utilizes 
the following construction. 
Let $(N_1,{\mathcal H}_1, K_1,G_1)$ define the quotient representation of the  CMC-1 surface equation in $\HT$ where ${\mathcal I}_{CMC}=\sysE_1/G_1$ is the Darboux integrable Pfaffian system whose integral manifolds correspond to constant mean curvature immersion. Then
let $(N_2,{\mathcal H}_2, K_2,G_2)$
 define the quotient representation for the Darboux integrable  Pfaffian system  ${\mathcal I}_{Liou} $ representing Liouville's equation given in \cite{FelsIvey} and so ${\mathcal I}_{Liou} =\sysE_2/G_2$. 
 We show in Theorem \ref{EquivT} that $K_1=K_2=K$, $G_1=G_2=G$ and that there exists  a $K$-equivariant biholomorphic map $\Psi :N_1\to N_2$ such $\Psi^* {\mathcal H}_2= {\mathcal H_1}$. This induces a diffeomorphism $\psi:N_1/G \to N_2 /G$ such that $\psi^* {\mathcal E_2}/G = {\mathcal E_1}/G$ giving rise to the commutative diagram
\begin{equation} 
\begin{gathered}
\tikzstyle{line} = [draw, -latex', thick]
\begin{tikzpicture}
\node(SLCP) {$\left(\, N_1 ,\, \sysE_1\,  \right)$};
\node[right of=SLCP, node distance=60mm](B){$\left(\, N_2  ,\, \sysE_2 \, \right)$};
\node[below of=B, node distance=20mm](I2) {$\left(\, N_2/G ,\, \sysE_2 /G\, \right)$};
\node[left of=I2, node distance=60mm](I1) {$\left(\, N_1/G ,\,\sysE_1/G \,\right)$};
\path[line](SLCP) -- node[above] {$\Psi $} (B);
\path[line](B) -- node[left] {$\pi_2$} (I2);
\path[line](I1) -- node[below] {$\psi$} (I2);
\path[line](SLCP) -- node[left] {${\pi}_{1}$} (I1);
\end{tikzpicture}
\end{gathered}
\label{CD0}
\end{equation}
where $G$ is a real form of $K$, $\pi_1$ and $\pi_2$ are the quotient maps,  and where $\psi$ defines an equivalence of the systems. 
In particular, in Theorem \ref{EquivT}
we have $K=SL(2,\C)$, $G=SU(2)$, $\sysE_1 /G={\mathcal I}_{CMC}$,
$\sysE_2 /G={\mathcal I}_{Liou}$, and diagram \eqref{CD0} in this case is given in equation \eqref{CD1T1}. 

The equivalence of these two systems leads to a Weierstrass-type representation for CMC-1 surfaces in $\HT$.
To make this explicit, we represent hyperbolic 3-space as the quadric hypersurface in $\R^4$ given by 
\begin{equation}\label{defofH3}
\HT= \{ (x^0, x^1,x^2,x^3) \ | \  (x^0)^2 -(x^1)^2 -(x^2)^2-(x^3) ^2 = 1, \ x^0 >0\ \}, 
\end{equation}
endowed with a Riemannian metric of constant sectional curvature $-1$ given by the restriction of the quadratic form
\begin{equation}\label{defofeta}
\eta= -(\rd x^0)^2 +(\rd x^1)^2 +(\rd x^2)^2+(\rd x^3) ^2
\end{equation}
(cf. \cite{BryantCMC}). 
Then we have the following:


\begin{thm}\label{LCMC} Let $U(x,y)$ be a solution to the Liouville equation \eqref{LiouEq} then 
\begin{equation}
\begin{aligned}
x^0 & = &&\frac{1}{16}\left( (1+x^2+y^2)(U_x^2+U_y^2)+8(xU_x+yU_y)+16\right) {\rm e}^{-\frac{U}{2}} +\frac{1}{4}(1+x^2+y^2){\rm e}^{\frac{U}{2}} \\
x^1 &=&& - \frac{1}{8}\left(x(U_x^2+U_y^2)+4 U_x)  \right){\rm e}^{-\frac{U}{2}} -\frac{1}{2}x {\rm e}^{\frac{U}{2}}\\
x^2&=&&  \frac{1}{8}\left(y(U_x^2+U_y^2)+4 U_y)  \right){\rm e}^{-\frac{U}{2}} +\frac{1}{2}y {\rm e}^{\frac{U}{2}}\\
x^3 &=&&\frac{1}{16}\left( (1-x^2-y^2)(U_x^2+U_y^2)-8(xU_x+yU_y)-16\right) {\rm e}^{-\frac{U}{2}} +\frac{1}{4}(1-x^2-y^2){\rm e}^{\frac{U}{2}} 
\end{aligned}
\label{IntroCMC}
\end{equation}
defines an immersed CMC-1 hypersurface in $\HT$.  Conversely, an immersed CMC-1 surface in $\HT$ can be written locally as the image of a solution to the Liouville equation under the map in equation \eqref{IntroCMC}. 
\end{thm}

The solutions to the CMC-1 given equation in Theorem \ref{LCMC} are obtained by applying the inverse of the map $\psi$ 
 in equation \eqref{CD1T1}  to solutions of the Liouville equation. The details are given in Corollary \ref{PT1}. Equation \eqref{IntroCMC} provides an alternative way to express the solutions to the CMC-1 equations in terms of  a holomorphic function by simply substituting for $U$ from equation \eqref{Liousol} in \eqref{IntroCMC} (or equation \eqref{RIM}). This produces an alternate Weierstrass-type representation of the solutions to the CMC-1 equation as given in \cite{LeviCMC}.


A similar result to Theorem \ref{LCMC} can be derived for CMC-1 surfaces in DeSitter space with $SU(2)$ being replaced with $SU(1,1)$ and the  Liouville equation \eqref{LiouEq} being replaced with the inequivalent ``minus'' form of the equation (see \cite{FelsIvey} for a discussion of the two inequivalent forms). 


\subsubsection*{Acknowledgements}
The authors would like to thank Ian Anderson, Andreas Malmendier and Peter Crooks of the department of mathematics and statistics at Utah  State University for useful discussions.

\subsection*{Group Action and Quotient Conventions}

Let $G$ act on the right on a set $X$ by $\mu:X\times G \to X$. Denote the equivalence classes of the orbits by $X/{\muG}$ and the projection $\pi_{\muG}: X \to X/{\mu}$. 
The orbit of $x\in X$ will be denoted $x \cdot G=\mu(x,G)$. 

Given a $G$-invariant function $f:X\to Y$ there exists a canonically defined function 
 $\lambda :X/{\mu}\to Y$ such that $f=\lambda \circ \pi_{\mu} $.  If $\lambda$ is a bijection then
we say that $f:X\to Y$ {\em represents} the quotient map. By an abuse of terminology we will often refer
to this as the quotient map, and denote it by $\Pi_{\mu}$.

Let $H\subset G$ be a closed Lie subgroup and let $G/H$ be the differentiable manifold of left cosets. We denote an
element of $G/H$ by $a \cdot H$ for $a\in G$, and we denote the quotient map to the space of cosets by $\pi_H: G \to G/H$. As above, we use $\Pi_H:G\to Y$ denote a representation of this quotient map.

An integrable distribution $\rchi^{}_{} \subset TM$ on the differentiable manifold $M$ defines the equivalence relation $x \cong y$ if $x$ and $y$ lie on the same maximal connected integral manifold of $\rchi^{}_{}$. We denote by $M/\rchi^{}_{}$ the set of equivalence classes and let $\pi_{\rchi^{}_{}}:M\to M/\rchi^{}_{}$ be quotient map to equivalence classes. Again we use $\Pi_{\rchi^{}_{}}:M \to Y$ for a representation of the quotient by an integral distribution. 

\subsection*{Differential Systems}

 We use the conventions in \cite{FelsIvey} for differential and Pfaffian systems. 
An {\defem exterior differential system} (EDS) on a smooth manifold $M$ is a graded ideal $\sysI$ 
(with respect to wedge product) inside the ring of differential forms on $M$, 
such that $\sysI$ is also closed under exterior differentiation.  
We assume that systems do not include 0-forms, and we assume that for $1\le k \le \dim M$ 
the $k$-forms of $\sysI$ span a vector sub-bundle $I_k$ of constant rank inside $\Lambda^k T^*M$.  

{\defem Pfaffian systems} are  generated by sections of a given sub-bundle $I=I_1$ of $T^*M$ and their exterior derivatives. (We will sometimes refer to $I$ itself as a Pfaffian system.)
The {\defem first derived system} $I'\subset T^*M$ of $I$ is spanned by sections $\theta$ of $I$ which satisfy $d\theta \in C^\infty( I \wedge T^*M)$.  When $M$ is a complex manifold, a {\defem holomorphic Pfaffian system} $\sysH$ is generated by holomorphic sections of a holomorphic sub-bundle $H \subset T^*_{1,0}M$. 

{\defem Monge-Amp\'ere systems} are systems $\sysI$ on a 5-dimensional manifold which are algebraically generated by a section of a given rank 1 sub-bundle $I_1 \subset T^*M$ (with Pfaff rank 5) and a pair of independent 2-forms spanning a rank 2 sub-bundle $I_2 \subset \Lambda^2 T^*M$.  See \cite{CfB2}. The CMC-1 and Liouville equations will be encoded by Monge-Amp\`ere systems while their prolongations are Pfaffian systems on 7-dimensional manifolds.



\section{Hyperbolic 3-Space and the CMC-1 System}
In this section we formulate the CMC-1 condition in $\HT$.  
Following Bryant \cite{BryantCMC} we use moving frames to describe an exterior differential system (EDS) with independence condition whose integral surfaces correspond to oriented surfaces in $\HT$ with constant mean curvature one.  In what follows, we take $\HT$ to be the hypersurface defined by \eqref{defofH3}, and use $\langle\, ,\, \rangle$ to denote the Minkowski inner product on $\R^4$ associated to the quadratic form $\eta$ in \eqref{defofeta}.

Let $\F$ denote the set of oriented bases $(\ve_0, \ve_1, \ve_2, \ve_3)$ of $\R^4$ which are orthonormal with respect to $\langle\, ,\, \rangle$. In detail, 
$\ve_0, \ve_1, \ve_2, \ve_3$ form the columns of a unimodular $4\times 4$ matrix, with 
$\langle \ve_i,\, \ve_j\rangle =\delta_{ij}$ for $i,j=1,2,3$, while 
$\langle \ve_i,\, \ve_0 \rangle=0$ and
$\langle \ve_0,\, \ve_0 \rangle=-1$. 
(Conventionally, $\ve_1, \ve_2, \ve_3$
are spacelike while $\ve_0$ is negative timelike.)
We also let $(\uve0, \uve1, \uve2, \uve3)$ denote the standard basis for $\R^4$, whose members satisfy the same relations.  
Again, by taking $\ve_0, \ve_1, \ve_2, \ve_3$ as columns of a matrix, it is easy to see that $\F$ is identified with $SO^+(3,1)$, the identity component of $SO(3,1)$.

For $0\le a,b \le 3$ define 1-forms $\omega^a, \omega^a_b = -\omega^b_a$ on ${\mathcal F}$ in the usual way:
\begin{equation}
\label{deees} \rd \ve_0 = \ve_a \omega^a, \qquad \rd \ve_a = \omega^a \ve_0 + \omega^b_a \ve_b, \qquad 1\le a,b \le 3.
\end{equation}
(The summation convention will be used throughout.) The 1-forms $\omega^1$, $\omega^2$, $\omega^3$
$\omega^2_1$, $\omega^3_1$, $\omega^3_2$ are pointwise linearly independent and define a coframe on $\F$.  Under the identification with $SO^+(3,1)$, these 1-forms are the components of the left-invariant Maurer-Cartan form, and satisfy structure equations
\begin{equation}\label{Fstructures}
\rd \omega^a = -\omega^a_b \wedge \omega^b, \qquad \rd \omega^a_b = -\omega^a_c \wedge \omega^c_b - \omega^a \wedge \omega^b,
\end{equation}
where $0\le c \le 3$ as well.  
For future reference, we let $\di_{\omega^a}$, $\di_{\omega^a_b}$ denote the dual left-invariant vector fields.

Note that the vector-valued function $\ve_0: \F \to \R^4$ takes value in $\HT$.  By taking $\ve_0$ as the $\R^4$-valued basepoint map, we see that $\F$ is the oriented orthonormal frame bundle of $\HT$. 
Any immersed oriented surface $s:M \to \HT$ equipped with a smooth oriented orthonormal frame $\ve_1, \ve_2$ for the tangent space to the surface gives a unique lift $\hat{s}$ to $\F$.  (For, if $\ve_1, \ve_2$ span the tangent spaces at $\ve_0$, there is a  unique $\ve_3$ -- the oriented unit normal to the surface -- such that $(\ve_0, \ve_1, \ve_2, \ve_3) \in SO^+(3,1)$; this gives the pointwise value of the lift.)  Since 
$\hat{s}^*\!\rd \ve_0$ is at every point in the span of $\ve_1, \ve_2$, by \eqref{deees} the image of $\hat s$ is an integral surface of the 1-form $\omega^3$ (i.e., $\hat{s}^* \omega^3=0$).  Since $s$ has rank two then $s^* \omega^1 \wedge\omega^2 \ne 0$ at each point of $M$.  

Of course, for a given immersed surface the choice of orthonormal frame is not unique because the orthonormal frame  can be modified by rotating $\ve_1, \ve_2$ through a smoothly-varying angle at each point.  Such rotations modify the lift $\hat{s}$ by motion along the orbits of the 1-parameter group generated by the vector field $\vecC=\di_{\omega^2_1}$.   In fact, under the identification $\F \cong SO(3,1)^+$ this action corresponds to right-multiplication by the subgroup $SO(2) \subset SO(3,1)^+$ defined by
\begin{equation}
SO(2)=\left\{ S_\theta = \left[ \begin{matrix} 1 & 0 &0 &0  \\ 0 & \cos \theta & - \sin \theta & 0 \\ 0 & \sin  \theta & \cos  \theta & 0  \\
0 & 0 & 0 & 1  
 \end{matrix} \right] \ , \ \theta \in [0, 2 \pi] \right\} 
\label{defSOC}
\end{equation}
and we refer to this as the {\em $SO(2)$-action} on $\F$.


We note that, since the structure equations imply that $\Lie_\vecC \omega^2_1=0$, the `horizontal' distribution $\Delta$  annihilated by $\omega^2_1$, 
\begin{equation}
\Delta=\{\, \omega^2_1\, \}^\perp= \span_{\R} \{ \partial_{\omega^1}, \partial_{\omega^2}, \partial_{\omega^3},\partial_{\omega^3_1}, \partial_{\omega^3_2} \} ,
\label{defADelta}
\end{equation}
is invariant under the $SO(2)$ action on $\F$.

Since 
$$\rd \omega^3 = -\omega^3_1 \wedge \omega^1 - \omega^3_2 \wedge \omega^2$$
from the structure equations, the fact that $\hat s^*\!\rd\omega^3=0$ implies that there are functions
$h_{ij}=h_{ji}$ for $1\le i,j \le 2$ such that $\hat s^*(\omega^3_i - h_{ij}\omega^j)=0$.  These $h_{ij}$
are the components of the second fundamental form of the immersed surface, and in particular the mean curvature is $\tfrac12 (h_{11} + h_{22})$.   It follows that if $M$ has mean curvature one then the image of the lift is an integral surface of the following EDS on $\F$
\begin{equation}
\sysM = \{ \omega^3, \omega^3_1 \wedge \omega^1 + \omega^3_2 \wedge \omega^2, \omega^3_1 \wedge \omega^2 - \omega^3_2 \wedge \omega^1 - 2 \omega^1 \wedge \omega^2 \}_\alg,
\label{CMCeds}
\end{equation}
with independence condition $\omega^1 \wedge \omega^2 \neq 0$.
Conversely, any integral manifold of $\sysM$ satisfying the independence condition defines a CMC-1 immersion into $\HT$ by applying the basepoint map $\ve_0$ to produce the immersed surface. 
The differential system $\sysM$ admits a Cauchy characteristic determined by the infinitesimal generator of the $SO(2)$ action on the right of $\F$ (see \eqref{defSOC}), so that the Cauchy characteristic distribution is 
\begin{equation}
\rchi_{\! \scriptstyle{\sysM}} = \span_\R \{ \ \vecC\ |\   \vecC=\di_{\omega^2_1} \  \} \ .
\label{chiM}
\end{equation}
We denote the reduced differential system 
\begin{equation}\label{defMcmc}
\tilsysMCMC = \sysM/\rchi^{}_{\!\! \scriptstyle{\sysM}} = \sysM/SO(2), 
\end{equation}
which is defined on the quotient manifold $\F/\rchi^{}_{\!\! \scriptstyle{\sysM}}=\F/SO(2)$.  (Since the 2-form $\omega^1 \wedge \omega^2$ is $SO(2)$-invariant,
the independence condition is well-defined on the quotient.) \footnote{Reduction of a system ${\mathcal I}$ on a manifold $M$ by Cauchy characteristics  $ \rchi^{}_{\!\scriptstyle{{\mathcal I}}}  $ see \cite{BCG3}, is a special case of \eqref{Qsys} which satisfies
$
\bq^* ({\mathcal I}/\rchi^{}_{\! \scriptstyle{{\mathcal I}}} ) = {\mathcal I}
$
where $\bq : M \to M/\rchi^{}_{\! \scriptstyle{{\mathcal I}}} $ is the quotient by the orbits of the maximal  connected integral manifolds of $\rchi^{}_{\!\scriptstyle{{\mathcal I}}} $. See Section 2.3 in \cite{AFCR}}
It is immediate from \eqref{CMCeds} that $\tilsysMCMC$ is a Monge-Amp\`ere system on the 5-dimensional quotient manifold $\F/\rchi^{}_{\!\! \scriptstyle{\sysM}}$.

\subsection{Elliptic Structure of the CMC System}

We now show that the differential system $\tilsysMCMC$ from equation \eqref{defMcmc}
is an elliptic decomposable differential system on $\F /SO(2)$
(see Definition 2.10 of \cite{FelsIvey}). Since $\tilsysMCMC$ is the reduction of $\sysM$ by its Cauchy distribution given in equation \eqref{chiM}, we will work with $\sysM$ on $\F$ from which this will follow by Cauchy reduction.

For convenience, we first introduce the complex-valued 1-forms
\begin{equation}
\momega = \omega^1 + \ri \omega^2, \qquad 
\meta = \omega^3_1 - \omega^1 + \ri( \omega^2 -\omega^3_2).
\label{defompi}
\end{equation}
It follows from \eqref{Fstructures} that these satisfy the structure equations
\begin{equation}
\begin{aligned}
\rd\omega^3 &=&& \frac{1}{2}(\momega \wedge  \meta  +\momegabar \wedge \metabar  )  ,\\
\rd \momega& = && ( \metabar +\momega)  \wedge \omega^3 - \ri \, \omega^2_1 \wedge \momega\, , \\
\rd \meta & = &&  (\omega^3 +\ri\, \omega^2_1 ) \wedge  \meta   . 
\end{aligned}
\label{EdsStreqs}
\end{equation}
In terms of these complex-valued 1-forms, the 2-form generators of $\sysM$ may be combined to give
$$
\omega^3_1 \wedge \omega^1 + \omega^3_2 \wedge \omega^2 + \ri (\omega^3_1 \wedge \omega^2 - \omega^3_2 \wedge \omega^1 - 2 \omega^1 \wedge \omega^2) \\
= \meta  \wedge \momega, 
$$
so that the complexified system with independence is 
\begin{equation}
\sysM \otimes \C =\{\ \omega^3,\ \meta \wedge \momega ,\ \metabar \wedge \momegabar\  \}_\alg,  \quad \momega \wedge \momegabar \neq 0. 
\label{eqMC}
\end{equation}

For a differential system $\sysI$ generated algebraically by 1-forms and 2-forms, the essential ingredient that makes $\sysI$ an elliptic decomposable differential system is a (sub-)complex structure $\J$ 
on the distribution $\calD$ annihilated by the 1-forms of $\sysI$, with respect to which the 2-form generators
of $\sysI \otimes \C$ can be taken to be pure $(2,0)$-forms and their complex conjugates. This structure determines and is uniquely determined by the splitting of $\calD \otimes \C$ into $-\ri$ and $+\ri$ eigenspaces of $\J$, labelled as  
$\calD_-$ and $\calD_+$ respectively; these are in turn determined by their annihilators which 
are {\em singular bundle} $V$ and its complex conjugate, respectively.
(The singular bundle is spanned by the 1-form generators of $\sysI$ and $(1,0)$-form factors of the $(2,0)$-forms; see Remarks 2.2 and 2.9 in \cite{FelsIveyArchive}).

In our case we give the system
$\tilsysMCMC$ on the 5-dimensional manifold $\F/SO(2)$ the structure of a decomposable
system by specifying its singular bundle.  (Notice that $\sysM$ can't be an elliptic decomposable system for dimensional reasons, since the distribution on the 6-dimensional manifold $\F$ that is annihilated by its single 1-form generator has odd rank.)  Noting that the generators of $\sysM \otimes \C$ can
be expressed in \eqref{eqMC} as decomposable 2-forms, and the factor $\momega$ of one of these 2-forms pulls back to the surface $M$ to be a $(1,0)$-form relative to the standard complex structure that rotates $\ve_1$ into $\ve_2$, we let 
\begin{equation} 
W = \spanC \{\, \omega^3,\, \meta, \momega\, \}.
\label{defW}
\end{equation}
With $\vecC$ as in equation \eqref{chiM},
\begin{equation*}
\Lie_\vecC \omega^3 = 0, \quad \Lie_\vecC \meta = \ri \meta, \quad \Lie_\vecC \momega = -\ri \momega.
\end{equation*}
Thus $W$ is  $SO(2)$ invariant and semi-basic for the action of $SO(2)$ on $\F$ and is hence  a basic sub-bundle. Moreover this also implies $\chi_M$ are Cauchy characteristics for $W$. 
Hence $W$ is the pullback via $\pi_{SO(2)}:\F \to \F/SO(2)$ of a well-defined rank 3 bundle 
$\tilV \subset T^* (\F/SO(2)) \otimes \C$ which is the singular system of $\tilsysMCMC$, and 
\begin{equation}
\tilV=W/\rchi^{}_{\! \scriptstyle{{\mathcal M}}} = W/SO(2)  \ . 
\label{defV}
\end{equation}

On the rank 4 distribution $\calD$ annihilating the 1-forms of $\tilsysMCMC$ 
on $\F/SO(2)$ it is possible to  define directly  the sub-complex structure by specifying it on a horizontal lift of $\calD$ given by
$$
\widehat\calD  = \span_{\R} \{ \partial_{\omega^1}, \partial_{\omega^2}, \partial_{\omega^3_1}, \partial_{\omega^3_2} \} = \Delta \cap \{\omega^3\} ^\perp
$$
where $\Delta$ is given in equation \eqref{defADelta}.  (Note that $\widehat\calD/\rchi^{}_{\! \scriptstyle{{\mathcal M}}}=\widehat\calD/SO(2)=\calD$.)  Letting 
$$
\widehat\calD_-  = 
(\Delta \otimes \C) \cap W ^\perp=\spanC
\{ \di_{\omega^3_1} - \ri\di_{\omega^3_2}, 
\di_{\omega^1} +\ri \di_{\omega^2} + \di_{\omega^3_1} + \ri\di_{\omega^3_2} \},$$
we see that this is the $-\ri$-eigenspace for the $\C$-linear extension of $\J:\widehat\calD \to \widehat\calD$ if and only if
$$
\begin{aligned}
\J( \partial_{\omega^3_1}) &= -\partial_{\omega^3_2}, & 
	\J( \partial_{\omega^1}+\di_{\omega^3_1}) &=\partial_{\omega^2}+ \partial_{\omega^3_2}, \\
\J( \partial_{\omega^3_2}) &= \partial_{\omega^3_1}, &
	\J( \partial_{\omega^2}+ \di_{\omega^3_2}) &= -\partial_{\omega^1}-\partial_{\omega^3_1}.
\end{aligned}
$$
The $SO(2)$-invariance of $W$ implies that $\widehat\calD_-$ and $\J$ are $SO(2)$-invariant and that these project by $\pi_{SO(2)}: \F \to \F/SO(2)$ to the corresponding objects for $\calD$.  Alternatively, a slight modification of Corollary 2.8 in \cite{FelsIveyArchive} can be used to show $\tilsysMCMC$ is an elliptic decomposable system.

An elliptic decomposable system is Darboux-integrable if the (complex) rank of $\tilV^{(\infty) } + \overline{\tilV}$ equals the (real) dimension of the underlying manifold.  The system $\tilsysMCMC$ fails to be Darboux-integrable since by equations \eqref{EdsStreqs}, \eqref{defW} and \eqref{defV}  
\begin{equation*}
\tilV^{(\infty)} =  W^{(\infty)}/ \rchi^{}_{\! \scriptstyle{{\mathcal M}}} = W^{(\infty)}/SO(2)= \spanC \{\; \meta\; \}/SO(2),
\end{equation*}
which has rank one while   $\overline{\tilV}$   has rank 3. In particular, the system $\tilsysMCMC$ has only one independent complex-valued Darboux invariant; see \eqref{deflam} and Lemma \ref{DI1def} below where we identify this invariant using a description of $\F/SO(2)$ as a manifold.


\subsection{Darboux Integrability of the Prolonged CMC-1 Equation for Surfaces in $\HT$}\label{DIPR}
Now we consider the prolongation of $\tilsysMCMC$.  


\begin{lem} \label{LPR} Let 
$\qpt \in \F/SO(2)$ and suppose $E \subset T_\qpt \F/SO(2)$ is a 2-dimensional  integral element of $\tilsysMCMC$ satisfying the independence condition. Then for each $\upt\in \Cq^{-1}_{SO(2)}(\qpt)$ there is a complex number $p$ such that
$$
\Cq_{SO(2)}^* E^\perp \vert_\upt =\spanR  \{ \omega^3_\upt, \Re(\meta_\upt-p \momega_\upt), \Im(\meta_\upt - \pbar \momega_\upt) \}.
$$
Moreover, $p$ varies smoothly along the fiber and satisfies $\Lie_{\vecC} p = 2\ri p$.
\end{lem}
\begin{proof} Let $\widehat E \subset T_\upt \F$ be any 2-plane such that $\pi_{SO(2)}{}_* \widehat E = E$.  Then
$\widehat E$ is an integral element of $\sysM$ satisfying the independence condition, so 
$\meta\vert_{\widehat E} = p \momega\vert_{\widehat E}$ for some complex number $p$.   
Moreover, if we modify $\widehat E$ by adding a vertical vector, $p$ is unchanged; hence $p$ a well-defined and smooth along the fiber.  Hence this equation implies
$$
\spanR\{ \omega^3_\upt, \Re(\meta_\upt-p \momega_\upt), \Im(\meta_\upt - \pbar \momega_\upt) \}
\subset \Cq^* E^\perp \vert_\upt,
$$
with equality since both sides have rank 3.  Along the fiber we compute
$$\Lie_{\vecC} (\meta - p \momega) = \ri \meta + (\ri p - \Lie_{\vecC} p) \momega.$$
But since $\Cq_{SO(2)}^* E^\perp$ is $SO(2)$-invariant, this must be a complex multiple of $\meta-p \momega$.  Hence $\Lie_{\vecC} p = 2\ri p$.
\end{proof}

\begin{remark} Suppose $s:M \to \F$ is an integral surface of $\sysM$ satisfying the independence condition. 
If we set $\omega^3_i = h_{ij} \omega^j$ for $1 \le i,j\le 2$ on the pullback to $M$, then $p$ pulls back to equal $h_{11} - \ri h_{12}-1$, and it follows that the $(2,0)$-part of the second fundamental
form of $M$ is $\tfrac12 p \momega \otimes \momega$.
\end{remark}

Matters being so, we let $\sysI$ be the real Pfaffian system on $\F \times \C$, with complex coordinate $p$ on the second factor, such that 
\begin{equation}
\begin{aligned}
\sysI \otimes \C &=&& \{ \omega^3, \meta - p \momega, \metabar  - \pbar \momegabar \}_\diff .
\end{aligned}
\label{prI}
\end{equation}
(By way of an abuse of notation, we pull 1-forms on $\F$ back to the product $\F \times \C$ without comment.) 
Since 
$$\rd\, (\meta - p \momega) \equiv (-\rd p + 2\ri p\, \omega^2_1) \wedge \momega$$
modulo the 1-forms of $\sysI \otimes \C$, the Cauchy characteristic distribution for $\sysI$ is
\begin{equation}
\rchi_{\scriptstyle{{\sysI}}} =\span_\R \{
\widehat \vecC  \mid \widehat \vecC= \di_{\omega^2_1}+2\ri (p\di_p - \pbar \di_{\pbar})\}.
\label{defchiI}
\end{equation}
The vector field $\widehat \vecC$ generates an extension of the $SO(2)$-action on $\F$ to $\F \times \C$, given by
\begin{equation}
\rho((F, p), S_\theta)= (F\cdot S_\theta,  {\rm e} ^ {2\ri \theta}p )
\label{defrho}
\end{equation}
where $F \in \F \cong SO(3,1)^+$, $S_\theta \in SO(2)$, and again we take $SO(2)$ to be the subgroup of matrices in $SO(3,1)^+$ given in equation \eqref{defSOC}.


Letting $\tilsysICMC=\sysI/\rchi^{}_{\scriptstyle{{\sysI}}}=\sysI/\rhoSO $ be the quotient
by this action,  
it then follows by Lemma \ref{LPR} that 
$\tilsysICMC$ is the prolongation of $\tilsysMCMC$.



As with our discussion of $\tilsysMCMC$ above, we give $\tilsysICMC$ the structure of an elliptic decomposable system by designating an $SO(2)$-invariant sub-bundle $\prW \subset T^* (\F \times \C) \otimes \C$ whose reduction $\prV=\prW/\rchi^{}_{\!\scriptstyle{{\sysI}}}=\prW/\rho$ is the singular bundle $\prV$ of $\tilsysICMC$:
\begin{equation}
\prW/\rchi^{}_{\!\scriptstyle{{\sysI}}} = \spanC \{\, \omega^3,\, \meta - p \momega, \metabar - \pbar \momegabar, \, \momega,\,
\rd p - 2\ri p \omega^2_1 \, \}/\rchi^{}_{\!\scriptstyle{{\sysI}}} = \prV
\label{defWpr}
\end{equation}
We computed the derived systems of $\prV$ from those of $\prW$, which are  
\begin{equation}
\begin{aligned}
&&\prW^{(1)}&= \spanC \{  \omega^3, \meta, \rd p-2\ri p \omega^2_1, \momega \}, \\
&&\prW^{(2)} &=\spanC  \{ \meta , \rd p-2p (\omega^3+\ri \omega^2_1), \momega \}\\
\prW^{(\infty)}&=&\prW^{(3)}&= \spanC \{ \meta , \rd p-2p (\omega^3+\ri \omega^2_1) \}
\end{aligned}
\label{flagWpr}
\end{equation}
Since the terminal derived system of $\prW$ (and hence $\prV$)  has rank two, $\tilsysICMC$ is Darboux-integrable. In the next section we will compute and interpret the Darboux invariants. 

We now make the following important observation based on Remark 3.3 in \cite{FelsIvey}, the Darboux integrability of $\tilsysICMC$, and the derived flag ranks $[5,4,3,2]$ of $\prW$ (and hence $\prV$).
\begin{corollary}\label{CQR} There exists (locally) a 5 dimensional complex manifold $N$ and holomorphic Pfaffian system $\sysH$ on $N$ with derived flag ranks $3,2,1,0$ with freely acting symmetry group $K$  such that $\tilsysICMC = \sysE/G$ where $\sysE$ is the real version of $\sysH$ and $G$ is a real form of $K$.
\end{corollary}
Locally there is only one possibility for $(N,\sysH)$ (see \cite{Yama}) and it is shown in Appendix \ref{A3} that $K$ is $SL(2,C)$ and $G$ is $SU(2)$. According to  Lie (see \cite{PurplePete}) there are only 3 possible local actions of $SL(2,\C)$ on $N$ and in Section \ref{findQR} we identify the $SL(2,\C)$ action for this case and give a global construction of $N$, $\sysH$, $K$ and $G$ satisfying Corollary \ref{CQR}.


Summarizing this section, we have the commutative diagram of differential systems   with horizontal maps being reduction and vertical being projection with $\tilde \pi_1$ being induced from $\pi_1:\F \times \C \to \F$ using the fact that the action $SO(2)$ in \eqref{defrho} is fiber preserving and giving,
\begin{equation} 
\begin{gathered}
\tikzstyle{line} = [draw, -latex', thick]
\begin{tikzpicture}
\node(SLCP) {$\left(\, \F \times \C ,\ {\sysI}\,  \right)$};
\node[right of=SLCP, node distance=60mm](B){
$\left(\, (\F\times \C)/\rchi^{}_{\!\scriptstyle{{\sysI}}},\  \tilsysICMC\, \right)$};
\node[below of=B, node distance=20mm](I2) {$\left( \, \F/\rchi^{}_{\!\! \scriptstyle{{\sysM}}}   ,\ \tilsysMCMC \,  \right)$};
\node[left of=I2, node distance=6 0mm](I1) {$\left( \, \F   ,\ {\mathcal M} \,  \right)$};
\path[line](SLCP) -- node[above] {$\pi_{\rchi^{}_{\!\scriptstyle{{\sysI}}}}$} (B);
\path[line](B) -- node[left] {$\tilde \pi_1$} (I2);
\path[line](I1) -- node[below] {$ \pi_{\rchi^{}_{\!\! \scriptstyle{{\sysM}}}} $} (I2);
\path[line](SLCP) -- node[left] {$\pi_1$} (I1);
\end{tikzpicture}
\end{gathered}
\label{findH1}
\end{equation}
where $\pi_{\rchi^{}_{\!\scriptstyle{{\sysI}}}}$ denotes the quotient map by the Cauchy characteristics of $\sysI$ whose maximal integral manifolds  are the $SO(2)$ orbits given by the action $\rho$ in equation \eqref{defrho} (so that  $\pi_{\rchi^{}_{\!\scriptstyle{{\sysI}}}}=\pi_\rho$).  We also have $ \pi_{\rchi^{}_{\!\! \scriptstyle{{\sysM}}}}=\pi_{SO(2)}$. 

Summarizing  from \eqref{eqMC}, \eqref{defW}, \eqref{prI}, \eqref{defWpr} the systems satisfy,
\begin{equation*}
\begin{aligned}
\sysI \otimes \C &=&& \{ \omega^3, \meta - p \momega, \metabar  - \pbar \momegabar \}_\diff , 
&\prW &=& &\spanC \{ \omega^3, \meta - p \momega, \metabar - \pbar \momegabar, 
\rd p - 2\ri p \omega^2_1, \momega \}\\
\sysM \otimes \C &=&& \{\ \omega^3,\ \meta \wedge \momega ,\ \metabar \wedge \momegabar\  \}_\alg ,
&W &=& &\spanC \{ \omega^3, \meta, \momega\},
\end{aligned}
\end{equation*}
and
\begin{equation}
\sysM =  \pi_{\rchi^{}_{\!\! \scriptstyle{{\sysM}}}}^*  \tilsysMCMC, \ 
W = \pi_{\rchi^{}_{\!\! \scriptstyle{{\sysM}}}}^* \tilV, \quad 
\sysI =\pi_{\rchi^{}_{\!\scriptstyle{{\sysI}}}} ^*\, \tilsysICMC\ , \ \prW=\pi_{\rchi^{}_{\!\scriptstyle{{\sysI}}}}^*\ \prV \ .
\label{defPB1}
\end{equation}
where $\tilsysICMC$ is the prolongation of $\tilsysMCMC$.


%% file: section3.tex
\section{The Lifted Systems on $SL(2,\C)$}

In this section we recall the action of $SL(2,\C)$ on $\R^4$ preserving the Minkowski inner product, which induces a double cover of the frame bundle $\F=SO(3,1)^+$ by $SL(2,\C)$.  Pulling systems $\tilsysMCMC$ and $\tilsysICMC$
back via this double cover will  simplify the computations as well as align them with references \cite{BryantCMC}, \cite{Lee}, \cite{Small}. The lifting will also allow us to determine the Darboux invariants as well as setting the stage for finding the quotient representation of $\tilsysICMC$ in Section \ref{findQR}, and then showing the equivalence to the elliptic Liouville system in Section \ref{findEQ}.

Following \cite{BryantCMC}, we  identify $\R^4$ with the vector space of Hermitian-symmetric $2\times 2$ matrices as follows:
\begin{equation}
\sfj: (x^0, x^1, x^2, x^3)^T \mapsto \begin{pmatrix} x^0 + x^3 &  x^1 + \ri x^2 \\ x^1 - \ri x^2 & x^0 - x^3 \end{pmatrix}
\label{defj}
\end{equation}
(see (1.6) in \cite{BryantCMC}).  Since this identification carries the Minkowski quadratic form to $-1$ times the determinant, the action of $SL(2, \C)$ on $\R^4$ defined by 
\begin{equation}\label{SL2onR4}
A \cdot {\bf x} = \sfj^{-1} \left( A \sfj({\bf x}) A^\dagger\right), \qquad A \in SL(2,\C),
\end{equation}
where $\dagger$ denotes the conjugate-transpose, preserves the Minkowski inner product (see (1.7) in \cite{BryantCMC}).
Again letting $\uve0, \uve1, \uve2, \uve3$ denote the standard basis vectors for $\R^4$, it is well-known that a double cover homomorphism $\sigma: SL(2,\C) \to \F=SO(3,1)^+$ is given by
\begin{equation*}
\sigma: A \mapsto (A \cdot \uve0, A \cdot \uve1, A \cdot \uve2, A \cdot \uve3) \in \F.
\end{equation*}


Composing $\sigma$ with the basepoint map $\ve_0$ gives a mapping 
\begin{equation}\label{SL2H3}
A \mapsto A \cdot \uve0 = \sfj^{-1} \left(A A^\dagger\right)
\end{equation}
which is surjective onto $\HT$.  Note that to check the second defining condition for $\HT$ in equation \eqref{defofH3} with 
\begin{equation}
A= \left[ \begin{matrix} a & b \\ c & d \end{matrix} \right]  \ , \quad ad-bc =1
\label{Adef}
\end{equation}
and writing $A \cdot \uve0 = (x^0,x^1,x^2,x^3) \in \R^4$ we have 
$x^0= \frac{1}{2}(|a|^2+|b|^2 +|c|^2+|d|^2) >0$.
Since the isotropy of $\uve0$ under the action \eqref{SL2onR4} is $SU(2)$, then 
\eqref{SL2H3} represents the  coset map  $\pi_{SU(2)}:SL(2,\C) \to SL(2,\C)/SU(2)$ 
and we denote this  by
\begin{equation}
\Pi_{SU(2)}(A) = \sfj^{-1}(  A A^\dagger).
\label{defpisu2}
\end{equation}
To summarize, the restriction of the $SL(2,\C)$ action \eqref{SL2onR4} on $\R^4$ to hyperbolic space 
realizes $\HT$ as an $SL(2,\C)$ homogeneous space, with isotropy $SU(2)$ at
$\uve0=(1,0,0,0)$.  This gives the diffeomorphism $\lambda : SL(2,\C)/SU(2) \to \HT$ such that $\Pi_{SU(2)}(A)= \lambda( A\cdot SU(2))$.

Since the center $\{ \pm I\}$ of $SL(2,\C)$ is the kernel of the homomorphism $\sigma$ and lies in $SU(2)$, there is a well-defined map $\Pi_{SO(3)}: SO(3,1)^+ \to \HT$ such that 
$$
\Pi_{SU(2)} = \Pi_{SO(3)} \circ \sigma \ .
$$
Here, $SO(3)$ is regarded as the subgroup of $SO(3,1)^+ \cong \F$ defined by 
\[ \left\{ \begin{bmatrix} 1 & 0 \\ 0 & M \end{bmatrix} \bigg \vert M \in GL(3,\R), M^{-1} = M^t \right \}\]
(this is also the fiber of $\F$ over $\uve0$). In fact  $\Pi_{SO(3)}$ is just the basepoint map on $\F$.
We use the notation $\Pi_{SO(3)}$ for the above map because it represents the
coset map $\pi_{SO(3)}: SO(3,1)^+ \to SO(3,1)^+/ SO(3)$.  (This follows from the fact
that $\sigma(SU(2)) = SO(3)$, so that $\sigma$ induces a natural diffeomorphism $SL(2,\C) / SU(2) \cong SO(3,1)^+/ SO(3)$.)

\medskip

For the rest of this article  we will make frequent reference to the following
subgroups of $SL(2,\C)$:
\begin{equation}
\begin{aligned}
\Borel &=&& \left\{\ \left[ \begin{matrix} a & b \\ 0 & a^{-1} \end{matrix} \right] \bigg \vert \ a\in \Cm, b \in \C \ \ \right\}\\
S^1&=&&\left\{ \left[ \begin{matrix} {\rm e}^{{\bf i} \theta} &0 \\ 0  &  {\rm e}^{-{\bf i} \theta} \end{matrix} \right] \bigg\vert \ \theta\in[0,2\pi] \ \right\} = SU(2) \cap \Borel .
\end{aligned}
\label{SO2sub}
\end{equation}



A relation between the three quotient spaces $SL(2,\C)/S^1, SL(2,\C)/SU(2)$ and $SL(2,\C)/\Borel$ is given in the next lemma.  For that we will use the representation $\Pi_\Borel:SL(2,\C) \to \CP^1 $ of the quotient map $\pi_\Borel:SL(2,\C)\to SL(2,\C)/\Borel$ given by the holomorphic function
\begin{equation}
\begin{aligned}
\Pi_\Borel: \begin{bmatrix} a & b \\ c & d \end{bmatrix}\mapsto [a,c], 
\label{defPiB}
\end{aligned}
\end{equation}
where the row vector on the right are homogeneous coordinates on $\CP^1$.

\begin{lem} 
The map $\Pi_{S^1}: SL(2,\C) \to \HT \times \CP^1$ given by
\begin{equation}
\Pi_{S^1}(A)= \left(\, \Pi_{SU(2)}(A),\, \Pi_{\Borel}(A)\,\right),
\label{defpis1}
\end{equation}
where $\Pi_{SU(2)}$ and $\Pi_{\Borel}$ are given in equations \eqref{defpisu2} and \eqref{defPiB} respectively, represents the quotient $\pi_{S^1}:SL(2,\C) \to SL(2,\C)/S^1$. 
\end{lem}

\begin{proof} We show, using a homogeneous space argument, that the canonical map as described in the Introduction, $\lambda:SL(2,\C)/S^1 \to \HT\times \CP^1$ defined by
\begin{equation}
\lambda(\pi_{S^1}(A) )  = (\, \Pi_{SU(2)}(A), \Pi_\Borel(A)\, ) =\Pi_{S^1}(A) \\
\label{deflam0}
\end{equation}
where $A \in SL(2,\C)$ is a bijection. To prove this we first we show that the diagonal left action of $SL(2, \C)$ on $\HT\times \CP^1$ is transitive. Any point in $(x,y) \in \HT \times \CP^1$ can be mapped to $\left((1,0,0,0),[1,0]\right)$ by an element of $SL(2,\C)$ acting diagonally by first mapping $x\to(1,0,0,0)$. Then using the $SU(2)$ isotropy of $(1,0,0,0)\in \HT$  we  map $y\to[1,0]$ which shows transitivity which implies $\HT\times \CP^1$ is an $SL(2,\C)$ homogeneous space for the diagonal left $SL(2,\C)$ action. 

The isotropy subgroup of $\left((1,0,0,0),[1,0]\right)=(\Pi_{SU(2)}(I_2) ,\Pi_{B}(I_2)) $ is the intersection $S^1$ given in equation \eqref{SO2sub}. The map $\lambda$ in equation \eqref{deflam0} is none other than the standard map representing $\HT\times \CP^1$  with basepoint $\left((1,0,0,0), [1,0]\right)$, as an $SL(2,\C)$ homogeneous space with $S^1$ isotropy, and as such is an $SL(2,\C)$ equivariant diffeomorphism. 
\end{proof}


The fact that $\sigma : S^1 \to SO(2)$ with $S^1$ and $SO(2)$ given in equations \eqref{SO2sub} and \eqref{defSOC}  leads to the following. 

\begin{corollary}
The quotient map $\Pi_{S^1}$ induces the $SO(2)$ invariant map $\Pi_{SO(2)}: SO(3,1) \to \HT \times \CP^1$ such that
\begin{equation} 
\begin{gathered}
\tikzstyle{line} = [draw, -latex', thick]
\begin{tikzpicture}%
\node(SLCP) {$SL(2,\C) $};
\node[below of=SLCP, node distance=20mm](I2) {$SO(3,1)^+$};
\node[right of=I2, node distance=50mm](B){$\HT \times \CP^1$};
\path[line](SLCP) -- node[left]  {$\sigma$} (I2);
\path[line](I2) -- node[below] {$\Pi_{SO(2)}$} (B);
\path[line](SLCP) -- node[above] {$\Pi_{S^1}$} (B);
\end{tikzpicture}
\end{gathered}
\label{CDSL2modS}
\end{equation}
is a commutative diagram and $\Pi_{SO(2)}$ represent the quotient map $\pi_{SO(2)}:\F\to \F/SO(2)$. 
\end{corollary}

\begin{proof}
The diffeomorphism $\lambda:SL(2,\C)/S^1\to \HT \times \CP^1$ in equation  \eqref{deflam} induces a diffeomorphism $\delta:SO(3,1)^+/SO(2) \to \HT\times \CP^1$ defined by
\begin{equation}
\delta \left( \sigma(A) \cdot SO(2)\right) = \Pi_{SO(2)}(\sigma(A)) = \lambda(A\cdot S^1)=\Pi_{S^1}(A)\ , \quad A\in SL(2,\C)
\label{deflam}
\end{equation}
so that $\Pi_{SO(2)}$ represents the quotient map $\pi_{SO(2)}$. 
\end{proof}






\subsection{Lifting  $\sysM$ to $SL(2,\C)$ and the first Darboux invariant} 
In this section we describe the pullback of $\sysM$ under the double cover $\sigma : SL(2, \C) \to \F$. We begin by specifying bases for the left-invariant vector fields and 1-forms on $SL(2,\C)$.
Let
\begin{equation}
E_1=\left[ \begin{matrix} 1 & 0 \\ 0 & -1  \end{matrix} \right], \ E_2 =\left[ \begin{matrix} 0 & 1 \\ 0 & 0  \end{matrix} \right], E_3=\left[ \begin{matrix} 0 & 0 \\ 1 & 0  \end{matrix} \right] 
\label{sl2basis}
\end{equation}
be a basis for $T_e SL(2,\C)$ and let $Z_i$ be the corresponding holomorphic left invariant vector fields on $SL(2,\C)$. On the open set dense subset of $SL(2,\C)$ given by $a\neq 0$ in equation \eqref{Adef}, these vector fields are given in the resulting coordinates $(a,b,c)$ by
\begin{equation}
Z_1  =   a \partial_a - b \partial_b +c \partial_c \ , \quad
Z_2  =  a \partial_b\ , \quad Z_3  =  b \partial_a +  \frac{1+bc}{a} \partial_ c .
\label{defZ}
\end{equation}
The vector fields $Z_i$ are the dual basis to the left invariant 1-forms $\alpha^i$ defined as components of the Maurer-Cartan form:
\begin{equation*}
A^{-1} dA =\left[ \begin{matrix} \alpha^1 & \alpha^2 \\ \alpha^3 & -\alpha^1 \end{matrix} \right] 
\end{equation*}
with $A$ given by equation \eqref{Adef}.  Either by computing the bracket relations of the $Z_i$, or directly from the coordinate expressions of the $\alpha_i$, we obtain the structure equations
\begin{equation}
\rd \alpha^1 = - \alpha^2 \wedge \alpha^3\ , \ \rd \alpha^2 = -2 \alpha^1 \wedge \alpha^2 \ , \ \rd \alpha^3 = 2 \alpha^1 \wedge \alpha^3 \ .
\label{dalphaeq}.
\end{equation}

Following \cite{BryantCMC} -- see equation (1.9) in that paper --
we express the Maurer-Cartan form as
\begin{equation}
\left[ \begin{matrix} \alpha^1 & \alpha^2 \\ \alpha^3 & -\alpha^1 \end{matrix} \right] = \tfrac12 \sigma^* \begin{pmatrix} \omega^3 + \ri \omega^2_1 & \metabar+2 \momega \\  - \meta & - (\omega^3 + \ri \omega^2_1)\end{pmatrix},
\label{gdg}
\end{equation}
where $\momega, \meta$ are as in \eqref{defompi}.
Let ${\mathcal B}=\sigma^* \sysM$.  Since \eqref{gdg} implies
\begin{equation}
\sigma^*\omega^3= \alpha^1+\overline{\alpha^1},\  \sigma^*\meta=-2 \alpha^3  \ , \ \sigma^* \momega= \alpha^2 +\overline{\alpha^3},
\label{pbmc}
\end{equation}
then from equations \eqref{eqMC} and \eqref{pbmc} we have
\begin{equation}
\begin{aligned}
{\mathcal B}\otimes \C = \sigma^*(\sysM\otimes \C) &=&&
\{\ \alpha^1+\overline{\alpha^1}, \ \alpha^3 \wedge ( {\alpha^2} + \overline{\alpha^3}),\ \overline{\alpha^3} \wedge ( \overline{\alpha^2} + {\alpha^3}) \ \}_\alg, 
\end{aligned}
\label{MAonSL2}
\end{equation}
while the independence condition pulls back to 
 $( {\alpha^2} + \overline{\alpha^3})\wedge ( \overline{\alpha^2} + {\alpha^3})\neq 0$.
The first structure equation in \eqref{EdsStreqs} gives the structure equation for $\sysB\otimes \C$
\begin{equation}
\rd\,  (\alpha^1 + \overline{\alpha^1}) =\tfrac{1}{2}\left(\alpha^3 \wedge ( {\alpha^2} + \overline{\alpha^3})+ \overline{\alpha^3} \wedge ( \overline{\alpha^2} + {\alpha^3})\right)=
 \tfrac{1}{2} (\Omega +\overline{\Omega}).
\label{dal1al1_}
\end{equation}
where  $\Omega=\alpha^3 \wedge ( {\alpha^2} + \overline{\alpha^3})$ and equations \eqref{dalphaeq} are used. This explicitly shows that $\sysB \otimes \C$ is algebraically generated as given in equation \eqref{MAonSL2}. We also define the singular bundle $\CCR$ associated to $\sysB$ such that $\sigma^* W = \CCR$.

Any integral surface  $s : M \to SL(2,\C )$  of ${\mathcal B}$ satisfying the independence condition defines a CMC-1 hypersurface in $\HT$ through composition with $\pi_{SU(2)}$ defined in \eqref{defpisu2}, and conversely, any CMC-1 hypersurface lifts (locally) to an integral manifold of $\sysB$. 

The system ${\mathcal B}$ in equation \eqref{MAonSL2} admits the Cauchy characteristic distribution
\begin{equation*}
\rchi^{}_{\! \scriptstyle{\sysB}} = \span_\R \{ \vecCY \ | \ \vecCY = \ri(Z_1 -\overline{Z_1}) \}
\end{equation*}
where  $\vecCY$ which is the infinitesimal generator of the one parameter subgroup $S^1$ in equation \eqref{SO2sub},
acting on $SL(2,\C)$ on the right. Note that as expected $\rchi^{}_{\scriptstyle{\sysB}} $ pushes forward by $\sigma$  to the Cauchy characteristic of $\sysM$ in equation \eqref{chiM},
\begin{equation}
\sigma_* \vecCY = 2 \partial_{\omega^2_1}.
\label{o12inZ}
\end{equation}

\def\sysMCM{{\sysM}_{CMC}}
\def\Vs{{V}}
\def\Vtinf{\Vs^{(\infty)}}
Let $\sysMCM$ be the Monge-Amp\`ere differential system on $\HT \times \CP^1$ defined by
$\delta^* \sysMCM = \tilsysMCMC$ where $\delta$ is the diffeomorphism given in equation \eqref{deflam}.  Immersed 2 dimensional integral manifolds of $\sysMCM$ (satisfying the independence condition) projected into $\HT$ are CMC-1 hypersurfaces.  (We also use diffeomorphism $\delta$ to define the singular bundle $\Vs$ for $ \sysMCM$ such  that $\delta^* \Vs = \tilV$.)
Then 
systems $\sysB$, $\sysM$ and $\sysMCM$ are related by the following commutative diagram, 
\begin{equation} 
\begin{gathered}
\tikzstyle{line} = [draw, -latex', thick]
\begin{tikzpicture}%
\node(SLCP) {$(\, SL(2,\C), \, \sysB \, )  $};
\node[below of=SLCP, node distance=20mm](T){$(\, SO(3,1)^+, \, \sysM \, )$};
\node[right of=T, node distance=60mm](B){$(\, \HT \times \CP^1, \, \sysMCM \, )$};
\path[line](SLCP) -- node[left] {$\sigma$} (T);
\path[line](SLCP) -- node[above] {$\ \ \Pi_{\rchi^{}_{\! \scriptstyle{\sysB}}} $} (B);
\path[line](T) -- node[below] {$\Pi_{\rchi^{}_{\!\! \scriptstyle{\sysM}}} $} (B);
\end{tikzpicture}
\end{gathered}
\label{CDEDSL1}
\end{equation}
which extends  \eqref{CDSL2modS} to include the corresponding systems and where the maps $\Pi_{\rchi^{}}$ are quotients by the corresponding Cauchy characteristics and where $\Pi_{\rchi^{}_{\! \scriptstyle{\sysB}}}= \Pi_{S^1}$ with $\Pi_{S^1}$ being given in equation \eqref{CDSL2modS}.
%
%
On account diagram \eqref{CDEDSL1}, and equations \eqref{defW} and \eqref{pbmc}, we note that for the singular system $\Vs$ of $\sysMCM$ the following holds,
\begin{equation}
\begin{alignedat}{4}
\Pi_{\rchi^{}_{\! \scriptstyle{\sysB}}}^* \Vs &= \CCR &= &\sigma^* W  
&&= \span_\C \{\, \alpha^1 + \overline{\alpha^1}, \,  \alpha^3,\, {\alpha^2} + \overline{\alpha^3} \,  \} \\
\Pi_{\rchi^{}_{\! \scriptstyle{\sysB}}}^* \Vtinf &= \Rinf &= &\sigma^* \Winf &&= \span_\C \{\, \alpha^3\, \} .
\end{alignedat}
\label{RandRinf}
\end{equation}


\medskip

We now identify the Darboux invariants for $\sysMCM$ by the following lemma.

\begin{lem} \label{DI1def} Let $\pi_2:\HT \times \CP^1\to \CP^1$ be the projection to the second factor in equation \eqref{deflam} and let $f:U \to \C$ be any holomorphic function defined on an open set $U \subset \CP^1$. Then $f \circ \pi_2 $  is a holomomorphic Darboux invariant for $\sysMCM$.
\end{lem}


\begin{proof} On $SL(2,\C)$ we have  
$\Rinf = \span_{\C} \{\, \alpha^3\, \}$ from \eqref{RandRinf}, and we will show that $d(  f \circ \pi_2 \circ \sigma ) \in \Rinf$, which would in turn imply by definition that $d(f \circ \pi_2)\in \Vtinf$. 

We first note that $\pi_2 \circ (\Pi_{SU(2)}\times \Pi_\Borel) = \Pi_\Borel:SL(2,\C) \to \CP^1$.   The vector fields  $Z_1$ and $Z_2$ are vertical for $\Pi_\Borel$ and  $\alpha^3(Z_1)=0, \alpha^3 (Z_2) = 0$ so that  $\alpha^3$ spans the $\Pi_\Borel$ semi-basic 1-forms on $SL(2,\C)$.

The function $f \circ \pi_2 \circ (\Pi_{SU(2)}\times \Pi_\Borel) = f \circ \Pi_\Borel $ is $\Pi_\Borel$-basic and holomorphic (since $\Pi_\Borel$ is holomorphic) on $SL(2,\C)$. 
Then
$d(f \circ \pi_2 \circ (\Pi_{SU(2)}\times \Pi_\Borel)) $ is $\Pi_\Borel$-basic and a holomorphic 1-form on $SL(2,\C)$, and is therefore a multiple of $ \alpha^3$ at each point, and thus lies in $\Rinf$.
\end{proof}

Alternatively, Lemma \ref{DI1def} can be stated as the pullback bundle identity $\Vtinf= \pi_2^*  (T_{1,0} \CP^1) $.



\subsection{Lifting the Prolongation.}
Define $\widehat\sigma: SL(2,\C)\to \F \times \C$ by $\widehat\sigma(A,z) = (\sigma(A), z)$.
let ${\sysJ}=\widehat{ \sigma}^* \sysI$ denote the pullback of $\sysI$ from $\F\times \C$  to $SL(2,\C)\times \C$. Computing using \eqref{prI} and \eqref{pbmc}, we have
\begin{equation}
\begin{aligned}
{\sysJ}\otimes \C &=&& \widehat \sigma^* (\sysI\otimes \C) \\ &=&&
\{\ \alpha^1+\overline{\alpha^1}, -2 \alpha^3 - p ( \alpha^2 + \overline{ \alpha^3}) , \, -2 \overline{\alpha^3} - \overline{p} ( \overline{\alpha^2} + { \alpha^3})  \ \} _{\diff}\, ,
\end{aligned}
\label{MAonSL2J}
\end{equation}
where $(\alpha^2 +\overline{\alpha^3})\wedge( \overline{\alpha^2} + { \alpha^3})$ is the independence condition.
This system has a Cauchy characteristic distribution $\rchi^{}_{\!\!\scriptstyle{{\sysJ}}}$ such that $\widehat \sigma_* \rchi^{}_{\!\!\scriptstyle{{\sysJ}}}= \rchi^{}_{\scriptstyle{{\sysI}}}$. Using equations \eqref{defchiI} and \eqref{o12inZ}, we obtain
\begin{equation}
\begin{aligned}
\rchi^{}_{\!\!\scriptstyle{{\sysJ}}}
& =&& {\rm span}_\R \{\frac{\ri}{2}(Z_1 -\overline{Z_1}) + 2\ri(p\partial_p -\pbar \partial_{\pbar}) \}\\
&=&&
 {\rm span}_\R \{\hatvecCY \ | \hatvecCY =\ri(Z_1-\overline{Z_1}) + 4\ri(p\partial_p -\pbar \partial_{\pbar}) \}.
 \end{aligned}
\label{chiI1}
\end{equation}
The infinitesimal generator of the action $p\mapsto \re^{4 \ri \theta} p $ is (with $p=r+\ri s$) is
$$
4 \ri (p \partial_p - \overline{p} \partial _{\overline p})
$$
Therefore the vector field $\hatvecCY$ in equation \eqref{chiI1}  is the infinitesimal generator of the action of $S^1$ in equation \eqref{SO2sub} on $SL(2,\C)\times \CP^1$ given by
\begin{equation}
\mu\left((A, p ); \left[ \begin{matrix} {\rm e}^{{\bf i} \theta} &0 \\ 0  &  {\rm e}^{-{\bf i} \theta} \end{matrix} \right]\right) \to  \left( A \left[ \begin{matrix} {\rm e}^{{\bf i} \theta} &0 \\ 0  &  {\rm e}^{-{\bf i} \theta} \end{matrix} \right], {\rm e}^{4{\bf i} \theta} p\right) . 
\label{S14A}
\end{equation}

Let $\pi_{\muS}: SL(2,\C) \times \C \to (SL(2,\C)\times \C) /\muS =(SL(2,\C)\times \C)/\rchi^{}_{\!\! \scriptstyle{{\sysJ}}}$ be the quotient map identifying the $S^1$ orbits of $\mu$.  To represent this quotient space we first let $\tau:(SL(2,\C) \times \C)\times \Borel \to SL(2,\C) \times \C$ be the action of the subgroup $\Borel$ in equation \eqref{SO2sub} given by
\begin{equation}
\tau((A,p), B_0 )= (A B_0, a^4 p), \quad\text{where } B_0=\left[ \begin{matrix} a & b \\ 0 & a^{-1} \end{matrix} \right] \in \Borel .
\label{deftau}
\end{equation}
As is well known (see for example \cite{wiki:Borel-Weil-Bott_theorem}) , $\Ofour \cong (SL(2,\C)\times \C)/\tau$ 
and so we let $\Pi_\tau: SL(2,\C) \times \C \to \Ofour$ represent $\pi_\tau:SL(2,\C) \times \C \to (SL(2,\C)\times \C)/\tau$.
We now have  the following lemma.

\begin{lem}  
The left $SL(2,\C)$ equivariant map $\Pi_{\muS}:SL(2,\C) \times \C \to \HT \times \Ofour$ given by
\begin{equation}
\Pi_{\muS}(A, p)=( \Pi_{SU(2)}(A), \Pi_{\tau}(A,p))
\label{Pimu}
\end{equation}
represents the quotient $\pi_{\muS}: SL(2,\C) \times \C \to (SL(2,\C)\times \C) /\muS$. 
\end{lem}

\begin{proof} In order to prove the lemma first note the $\Pi_{\muS}$ is $\mu$ invariant, and then we show that  the canonical map
$\widehat \lambda:  (SL(2,\C)\times \C) /\muS \to \HT \times \Ofour$ given by
\begin{equation}
\begin{aligned}
\widehat \lambda\left(\, \bq_{\muS}(A,p))\,  \right) &=&&\left( \Pi_{SU(2)}(A) , \Pi_{\tau}(A,p)  \right) \\
\end{aligned}
\label{defhlam1}
\end{equation} 
is a bijection.
For injectivity, suppose  $\widehat \lambda(\bq_{\muS}( A,p)) = \widehat \lambda (\bq_{\muS}(Y,q))$; 
we show $ \bq_{\muS}( A,p)=\bq_{\muS}( Y,q)$. In other words, there exists $S_\theta\in S^1$ such that $(Y,q) = (AS_\theta, {\rm e}^{4{\bf i} \theta} p)$. The condition $\Pi_{\tau}( Y, q) = \Pi_{\tau}(A,p)$ implies by equation \eqref{deftau} that there exists $B_0 \in B$ such that
\begin{equation}
(Y,q) = ( A B_0,   a^4  p)
\label{Eq1}
\end{equation}
with $B_0$ as in equation \eqref{deftau}. The condition $\Pi_{SU(2)}(Y)=\Pi_{SU(2)}(A)$ implies that
\begin{equation}
Y  = A C \text{ for some } C \in SU(2).
\label{Eq2}
\end{equation}
In order for equations \eqref{Eq1} and \eqref{Eq2} to hold we have $B_0=C \in SU(2) \cap B = S^1$. Therefore with $B_0\in S_1$, equation \eqref{Eq1} shows that  $\bq_{\muS}( A,p)=\bq_{\muS}( Y,q)$.

In order to show that $\widehat \lambda$ is surjective, we show that $\Pi_{\muS}$ is surjective. Let $h_0\in \HT$ and let $A_0\in SL(2,\C)$ satisfy $\Pi_{SU(2)}(A_0) = h_0$. Then $\Pi_{SU(2)}(A_0U) =h_0$ for all $U\in SU(2)$. We then claim the mapping $\widehat\Pi_{\tau}: SU(2) \times \C \to \Ofour$, defined by
$$
\widehat\Pi_{\tau}(U,p) =\Pi_{\tau}(A_0U,p) ,\qquad U \in SU(2), \ p \in \C,
$$
is surjective. Since $\Pi_{\tau}$ is equivariant $\Pi_{\tau}(A_0U,p)= A_0 \cdot \Pi_{\tau}(U,p)$ and
 $\Pi_{\tau}$ restricted to $SU(2)\times \C$ is surjective (being the representation of $\Ofour$ as an $S^1$ quotient of $S^3\times \C$) 
 we have $\widehat\Pi_{\tau}$ is surjective. Therefore $\widehat \lambda$ is surjective.
\end{proof}

Since the orbits of $\mu$ are the maximal connected integral manifolds of $\rchi^{}_{\!\!\scriptstyle{{\sysJ}}}$ then  
$\Pi_{ \rchi^{}_{\!\! \scriptstyle{{\sysJ}}}}=\Pi_\muS$ represents $\pi_{ \rchi^{}_{\!\!\scriptstyle{{\sysJ}}}}$. This also produces a representation of $\pi_{ \rchi^{}_{\scriptstyle{{\sysI}}}}$.

\begin{corollary} The map $\Pi_{ \rchi^{}_{\scriptstyle{{\sysI}}}}: \F \times \C \to \HT \times \Ofour$ defined by requiring 
$\Pi_{ \rchi^{}_{\scriptstyle{{\sysI}}}} \circ \widehat \sigma = \Pi_{ \rchi^{}_{\!\!\scriptstyle{{\sysJ}}}}$ 
is a representation of $\pi_ {\rchi^{}_{\! \scriptstyle{\sysI}}}: \F \times \C \to (\F \times \C) / \rchi^{}_{\! \scriptstyle{\sysI}}$.
\end{corollary}


The system $\tilsysICMC=\sysI/\rchi^{}_{\! \scriptstyle{\sysI}}$ on $(\F \times \C) / \rchi^{}_{\! \scriptstyle{\sysI}}$ was shown in Section \ref{DIPR} to be a Darboux integrable elliptic decomposable system. The system $\sysICMC$ on $\HT\times \Ofour$ defined by $ \widehat \lambda ^* \sysICMC= \tilsysICMC$ with the diffeomorphism $ \widehat \lambda $ in equation \eqref{defhlam1} of course has the same properties. Furthermore since $\tilsysICMC$ is the prolongation of  $\tilsysMCMC$, we have $\sysICMC$ is the prolongation of $\sysMCMC$.  The corresponding singular system $V$ for $\sysICMC$ is defined by $\widehat \lambda^* V = \prV $.

Let $\widehat{R} = \Pi_{ \rchi^{}_{\!\!\scriptstyle{{\sysJ}}}}^* \widehat{V}$ which on account of 
$\Pi_{ \rchi^{}_{\!\!\scriptstyle{{\sysJ}}}}=\Pi_{ \rchi^{}_{\scriptstyle{{\sysI}}}}\circ\sigma$ and equation \eqref{defPB1} satisfies $\widehat{R}  = \widehat \sigma^* \widehat{W}$.  
Therefore from equation \eqref{defWpr} and \eqref{pbmc}, 
$$
\widehat{R} = \widehat \sigma^* \widehat{W} =\spanC \{  \alpha^1 + \overline{\alpha^1}, \,  
-2 \alpha^3-p( {\alpha^2} + \overline{\alpha^3}),
-2 \overline{\alpha^3}-\pbar( \overline{\alpha^2} + {\alpha^3}),
 {\alpha^2} + \overline{\alpha^3}, \rd p-2 p(\alpha^1-\overline{\alpha^1})\}.
$$
We also have from equation \eqref{flagWpr} and \eqref{pbmc},
\begin{equation}
\prRinf=\sigma^* \prW^{(\infty)}= \Pi_{ \rchi^{}_{\!\!\scriptstyle{{\sysJ}}}}^*\widehat{V}^{(\infty)} =  \spanC \{ \alpha^3 ,  \rd p -4p \alpha^1 \}
\label{defhatRinf}
\end{equation}
which will allow us to determine the Darboux invariants for $\sysICMC$ in a similar way to those of $\sysMCMC$ in Lemma \ref{DI1def}.

\begin{lem} 
Let $\pi_2:\HT \times \Ofour\to \Ofour$ be the projection to the second factor in equation \eqref{deflam} and let $f:U \to \C$ be a holomorphic function defined on an open set $U \subset \Ofour$. Then $f \circ \pi_2 $  is a holomomorphic Darboux invariant for $\sysICMC$.
\end{lem}

\begin{proof} As in the proof in Lemma \ref{DI1def} we work on $SL(2,\C) \times \C$ and 
show that $f \circ \pi_2$ is a Darboux invariant for $\sysICMC$ if and only if the $(1,0)$ form $d(f \circ \pi_2 \circ (\Pi_{SU(2)}\times \bq_{\tau} )$  on $SL(2,\C)\times \C$
lies in  $\prRinf$.

Since $\pi_2 \circ (\Pi_{SU(2)}\times \bq_{\tau})  = \bq_{\tau}$ we have 
$f \circ \pi_2 \circ (\Pi_{SU(2)}\times \bq_{\tau})=f \circ \bq_{\tau}$. On the other hand the bundle $\prRinf$ in equation \eqref{defhatRinf}
is $\bq_{\tau}$ semi-basic, and so $d(f \circ \bq_{\tau} )$ lies in $\prRinf$ at each point.
\end{proof}


We now summarize the relations between our various differential systems by the following commutative
diagram which provide an explicit representation of  the diagram \eqref{findH1} augmented with the maps $\widehat \sigma$ and $\sigma$,
\begin{equation*} 
\begin{gathered}
\tikzstyle{line} = [draw, -latex', thick]
\begin{tikzpicture}
\node(SLCP) {$\left( SL(2,\C)\! \times\! \C ,\sysJ,   \prR  \right)$};
\node[right of=SLCP, node distance=55mm](T1){$\left( \F\! \times\! \C   , \sysI, \prW   \right)$};
\node[right of=SLCP, node distance=113mm](T2){$\left( \HT\! \times\! \Ofour  , \sysICMC ,\widehat V\, \right)$};
\node[below of=SLCP, node distance=20mm](B1) {$\left( SL(2,\C) ,{\mathcal B},  \CCR \right)$};
\node[below of=T1, node distance=20mm](B2) {$\left( \F  , \sysM ,W \right)$};
\node[below of=T2, node distance=20mm](B3) {$\left( \HT\times \CP^1  , \sysM_{CMC}, V  \right)$};

\draw [->, line width=0.3mm] (SLCP) to [out=20, in=160] node[midway, sloped, above] {$\Pi_ {\rchi^{}_{\!\!\scriptstyle{\sysJ}}}$}(T2);
\draw [->, line width=0.3mm] (B1) to [out=340, in=200] node[midway, sloped, above] {$ \Pi_{\rchi^{}_{\! \scriptstyle{\sysB}}}$}(B3);

\path[line](T1) -- node[below] {$\Pi_{\rchi^{}_{\! \scriptstyle{\sysI}}}$} (T2);
\path[line](B2) -- node[below] {$\Pi_{\rchi^{}_{\!\! \scriptstyle{\sysM}}}$} (B3);

\path[line](SLCP) -- node[below] {$\widehat \sigma
$} (T1);
\path[line](SLCP) -- node[left] {${\pi_1}$} (B1);
\path[line](T1) -- node[left] {$\pi_{1}$} (B2);
\path[line](T2) -- node[left] {${\rm Id} \times {\bf q}$} (B3);
\path[line](B1) -- node[below] {$\sigma$} (B2);
\end{tikzpicture}
\end{gathered}
\end{equation*}
The map ${\bf q}: \Ofour \to \CP^1$ is the standard bundle projection map while all the maps of the form $\Pi_{\rchi^{}}$ are reduction by the corresponding Cauchy characteristics. For convenience we also note that 
$\Pi_{\rchi^{}_{\!\scriptstyle{\sysB}}}=\Pi_{S^1}$ with $\Pi_{S^1}$ given in equation \eqref{CDSL2modS} and that $\Pi_{\rchi^{}_{\!\! \scriptstyle{\sysJ}}}=\Pi_{\mu}$ with $\Pi_{\mu}$  in equation \eqref{Pimu}.


%% file: section4.tex
\section{The Quotient Representation of $\sysMCMC$ and ${\mathcal I}_{CMC}$.}\label{findQR}

In Theorem A of Bryant's original paper \cite{BryantCMC}, he shows how to lift integral manifolds of $\sysMCMC$ (or $\tilsysMCMC$) which correspond with CMC-1 surfaces in $\HT$  to holomorphic immersed curves in $SL(2,\C)$ whose tangent space lies in the null cone of the Killing form using an equation of Lie type for SU(2).  The  process of lifting integral manifolds using equations of Lie type is used in the so-called reconstruction problem for lifting integral manifolds in quotient systems, see Section 6 in \cite{AF1}. In particular, lifting integral manifolds from a quotient system $\sysI/G$ to the original system $\sysI$ requires solving a differential system of Lie type on the Lie group $G$. We now identify the differential system alluded to in Bryant \cite{BryantCMC} for which $\sysM_{CMC}$ is the $SU(2)$ quotient.

\subsection{The Quotient Representation}

Let  ${\mathcal H}$ be the holomorphic  rank 2 (over $\C$) Pfaffian system on $SL(2,\C) \times \CP^1$ given in
homogeneous coordinates $(p,q) $ for $\CP^1$  by
\begin{equation}
\sysH= {\it span}_\C \{  \ pq \alpha^2 - q^2 \alpha^1, \  pq \alpha^3 +p^2 \alpha^1,\  q^2 \alpha^3 +p^2 \alpha^2\ \} \ .
\label{Hdef}
\end{equation}
On the open set $U =\{ (A,[p,q]) \in SL(2,\C)\times \CP^1 , \ q\neq 0 \}$ let $u= pq^{-1}$ be the corresponding local coordinate on $\CP^1$. The system $\sysH$ from equation \eqref{Hdef} then takes the form,
\begin{equation}
{\mathcal H}\vert_U  = \spanC \{ \eta^1= \alpha^1 - u \alpha^2,\ \eta^2= \alpha^3 + u^2  \alpha^2 \}_{{\rm diff}}.
\label{HU}
\end{equation}
The Pfaffian system $\sysH$ is invariant under the diagonal free right action of $SL(2,\C)$ on $SL(2,\C) \times \CP^1$ given by
\begin{equation}
\Delta( (A,[p,q]) ; C) = (AC, [p,q] C ) , \quad C\in SL(2,\C). 
\label{SL2action}
\end{equation}

The construction of the system $\sysH$ is given in Appendix \ref{A2} where it is shown (see Theorem \ref{defsysH})
that the holomorphic integral curves of $\sysH$ which are transverse to the projection $\pi:SL(2,\C) \times \CP^1\to SL(2,\C)$ 
 project to holomorphic curves in $SL(2,\C)$ whose tangent vector lie in the null cone with respect to the Killing form on $T_{1,0}SL(2,\C)$ as in \cite{BryantCMC}.  This leads to have the quotient representation theorem for $\sysMCMC$. 

\begin{thm} \label{TMAreduction} Let $\sysE$ be the real verson of ${\mathcal H}$ on $SL(2,\C)\times \CP^1$ in equation \eqref{Idef}, and let  $\Delta_{SU(2)}$ be the restriction of the right diagonal action of $SL(2,\C)$ in equation \eqref{SL2action} to $SU(2)$. Then
\begin{equation}
\sysE/{\Delta_{SU(2)}} =  \sysM_{CMC}.
\label{MAreduction}
\end{equation}
\end{thm}


\begin{proof}  In order to prove this theorem we represent the quotient map by $\Pi_{\Delta_{SU(2)}}: SL(2,\C) \times \CP^1 \to \HT \times \CP^1$ as in equation \eqref{PiAp} from Lemma \ref{Lpsimap}. The action $\Delta_{SU(2)}$  of $SU(2)$ is easily checked to be transverse to real rank 4 Pfaffian system $\sysE$. The system $\sysE$ is generated by 4 independent 1-forms and a pair of independent 2-forms,
and so by Corollaries 7.3 and 7.4  in \cite{AF1}, the quotient system $\sysE/\Delta_{SU(2)}$ is generated algebraically by a one-form and a pair of 2-forms. Since ${\mathcal M}_{CMC}$ has this property, the proof of \eqref{MAreduction} requires we show
\begin{equation}
\Pi_{\Delta_{SU(2)}} ^*  ({\mathcal M}_{CMC} ) \subset \sysE .
\label{PBP1}
\end{equation}
In order to verify equation \eqref{PBP1}, we need use the following characterizing property of $\sysM_{CMC}$,
\begin{equation*}
\Pi_{S^1}^* ({\mathcal M}_{CMC} ) =\sysB ,
\end{equation*}
where $\Pi_{S^1}:SL(2,\C) \to \HT \times \CP^1$ is given in equation \eqref{defpis1} and the $\Pi_{S^1}$ quotient is by Cauchy characteristics $\rchi^{}_{\! \scriptstyle{{\sysB}}}$.

Therefore we need to find a map $\zeta: SL(2,\C)\times \CP^1 \to SL(2,\C)$ such that
\begin{equation}
\Pi_{\Delta_{SU(2)}}= \Pi_{S^1} \circ \zeta
\label{simpPi}
\end{equation}
and 
\begin{equation}
\zeta ^* \sysB \subset \sysE .
\label{PBP2}
\end{equation}
With $\zeta$ satisfying equations \eqref{simpPi} and \eqref{PBP2},  we have 
$$\Pi_{\Delta_{SU(2)}}^* \sysMCMC = \zeta^*\circ \Pi_{S^1}^* \ \sysMCMC = \zeta^*\ \sysB \subset \sysE.$$ 
Thus equations \eqref{simpPi} and \eqref{PBP2}  holding proves Theorem  \ref{TMAreduction}. 

In terms of a diagram we construct a function $\zeta $  to make the commutative diagram 
\begin{equation} 
\begin{gathered}
\tikzstyle{line} = [draw, -latex', thick]
\begin{tikzpicture}
\node(SLCP) {$  SL(2,\C) \times \CP^1  $};
\node[right of=SLCP, node distance=80mm](B) {$\left(SL(2,\C), \sysB=\Pi_{S^1}^* \sysMCMC \right)$};
\node[right of=I1, node distance=80mm](I2) {$\left(\HT \times \CP^1, \sysMCMC \right)$};
\path[draw , dashed,->](SLCP) -- node[above] {$\zeta $} (B);
\path[line](B) -- node[right] {$ \Pi_{S^1}$} (I2);
\path[line](SLCP) -- node[right] {\raisebox{10pt}{$\Pi_{\Delta_{SU(2)}}$}} (I2);
\end{tikzpicture}
\end{gathered}
\label{CD1}
\end{equation}
in order to compute $\Pi_{\Delta_{SU(2)}}^* \sysMCMC$.

A function $ \zeta$ can be constructed by using equation \eqref{defgam} (or equation \eqref{defalpha}) in Lemma \ref{PreLpsimap} which is used to construct the the quotient map $\Pi_{\Delta_{SU(2)}}$  in Lemma \ref{Lpsimap}. To be specific, let $U=\{ (A,[p,q] ) \in SL(2,\C) \times \CP^1 \ , \ q \neq 0\}$ and let $u=pq^{-1}$ be the corresponding local coordinate on $\CP^1$. We then let $\zeta:U \to SL(2,\C)$ be 
$$
\zeta(A,u) = A R(u)
$$
where  $R(u)$ satisfies $ [u,1] \cdot R(u) = [0,1], \ R(u) \in SU(2)$ on $U$. This is the $AR$ term in the right hand side of equation \eqref{defgam} (or equation \eqref{defalpha}). Solving$ [u,1] \cdot R(u) = [0,1]$ for $R(u)$ gives,
\begin{equation}
 \zeta( A, u ) =  A \ \left( \frac{1}{\sqrt{1+|u|^2}} 
 \left[ \begin{matrix} 1 & \bar u \\ -u & 1 \end{matrix} \right]\ \right).
 \label{deltamap}
 \end{equation}
The proof of Lemma \ref{Lpsimap} shows $\Pi_{\Delta_{SU(2)}}(A,u) = \Pi_{S^1}\circ \zeta(A,u)$ on $U$. Hence condition \eqref{simpPi} is satisfied.

We now show $\zeta^* \sysB \subset \sysE$ (equation \eqref{PBP2}) holds where $\zeta $ is given in equation \eqref{deltamap}, $\sysB$ is determined by equation \eqref{MAonSL2}, and $\sysE$ is determined by the real version of $\sysH$ from equation \eqref{HU}. Computing using equations \eqref{deltamap} and \eqref{gdg}  we have
 $$
 \zeta^* (\alpha^1 + {\overline \alpha^1}) = \frac{1-|u|^2}{1+|u|^2}( \eta^1 +\overline{\eta^1}) - \frac{1}{1+|u|^2}( \bar u \eta^2 + u \overline{\eta^2} ).
$$
where $\eta^1$ and $\eta^2$ are given in equation \eqref{HU}.  Thus $\zeta^*  (\alpha^1 + {\overline \alpha^1})\in \sysE$.  Since $\zeta^* \rd \, (\alpha^1 +{\overline \alpha^1})\in \sysE$ equation \eqref{dal1al1_} shows $\zeta^*(\Omega+\overline{\Omega}) \in \sysE$ where $\Omega= \alpha^3 \wedge (\alpha^2 + \overline{\alpha^3})$.

To finish showing $\zeta^* \sysB \subset \sysE$ we need to show $\zeta^*(\ri(\Omega-\overline{\Omega})) \in \sysE$. Using $\Omega= \alpha^3 \wedge (\alpha^2 + \overline{\alpha^3})$, and equations \eqref{gdg}, \eqref{deltamap}, we find
\begin{equation}
\zeta^*(\ri( \Omega-\overline{\Omega} ))= \ri(-2 \rd u \wedge \alpha^2  -2 \rd \bar u \wedge \overline{\alpha^2})     \qquad \mod \ \sysE \ .
\label{ppo1}
\end{equation}
Equation \eqref{HU} can be written
\begin{equation*}
\rd  \eta^1 = - \rd u \wedge \alpha^2 \qquad \mod \sysH
\end{equation*}
which together with equation \eqref{ppo1} gives
$$
\zeta^* ({\ri} (\Omega-  \overline{\Omega})) = 2{\ri}  \rd\, (\eta^1 - \overline {\eta^1}) \quad \mod  \sysE .
$$
Since ${\ri}  \rd \, (\eta^1 - \overline {\eta^1}) \in \sysE$, we have $\zeta^* ({\ri} (\Omega-  \overline{\Omega})) \in \sysE$ and so $\zeta ^* \sysB \subset \sysE$ and  $ \sysE/\Delta_{SU(2)} = {\mathcal M}_{CMC}$.
 \end{proof}

By prolongation $ \sysICMC = pr(\sysE)/{\widehat \Delta}_{SU(2)}$ where the prolongation $pr(\sysE)$ and the diagonal action $\widehat \Delta$ of $SU(2)$  are defined on $SL(2,\C) \times \Ofour$. The system  $pr(\sysE)$ is the real version of the prolongation of $\sysH$ in equation \eqref{HU} which is given locally by 
$$
pr( \sysH)\vert_U= \{\ \alpha^1 - u \alpha^2,\  \alpha^3 + u^2  \alpha^2  ,\ du - s \alpha^2\ \}
$$
where $U$ us the open set $(A,u,s)\subset SL(2,\C)\times \Ofour $ with $s$ the fiber coordinate of $\Ofour$.  The map $\widehat \zeta:U \to SL(2,\C) \times \C$ (extending $\zeta$ above) is given in coordinates by,
$$
\widehat \zeta( A, u, s ) = \left(\zeta(A,u), p= \frac{2s}{(1+|u|^2)^2} \right)
$$
and $\widehat \zeta^* \sysBpr\subset \sysEpr$. In otherwords $\sysEpr/{\widehat \Delta}_{SU(2)} = \sysMCMCpr$

\medskip

\section{Equivalence of CMC-1 equation to Liouville's equaiton}\label{findEQ}

The Pfaffian system ${\mathcal H}$ in equation \eqref{Hdef} has derived system
\begin{equation}
\sysH ' = {\rm span}_{\C } \{ \theta = \ 2pq \alpha^1 -p^2 \alpha^2 + q^2 \alpha^3 \ \}  ={\rm span}_{\C } \{ 2u \alpha^1 - u^2 \alpha^2 + \alpha^3 \}_{u=pq^{-1}}.
\label{defHp}
\end{equation}
leading to the ranks of the derived flag of $\sysH$ being $[2,1,0]$. 
According to  \cite{Yama}, $\sysH$ is (locally) equivalent to $(J^2(\CP^1 , \CP^1 ),{\mathcal C})$ where $J^2(\CP^1 , \CP^1 )$ is the space of 2-jets of maps from $\CP^1$ to $\CP^1$ and ${\mathcal C}$ is the standard contact system \cite{BCG3}.

The derived system $\sysH'$ admits the complex Cauchy characteristic distribution,
\begin{equation}
\rchi^{}_{\scriptstyle{{\mathcal H}'}}= {\rm span}_{\C} \{ pq  P_1 + q^2 P_2  -p^2  P_3 \}
\label{defchi}
\end{equation}
where 
$$
P_1 = Z_1 + p \partial_p - q\partial_q, \quad
P_2  =  Z_2+p\partial_q \quad
P_3  =  Z_3 + q \partial_p
$$
are the infinitesimal generators of $SL(2,\C )$ acting diagonally on the right on $SL(2,\C ) \times \CP^1$ corresponding to the $E_i$ basis in equation \eqref{sl2basis}. The quotient of the derived system $\sysH'$ by $\rchi^{}_{\scriptstyle{{\mathcal H}'}}$ makes $(SL(2,\C)\times \CP^1)/\rchi^{}_{\scriptstyle{{\mathcal H}'}}$ a complex 3-dimensional contact manifold with holomorphic contact structure.

The action of $SL(2,\C)$ on $SL(2,\C ) \times \CP^1$ projects to $(SL(2,\C)\times \CP^1)/\rchi^{}_{\scriptstyle{{\mathcal H}'}}$ and preserves the contact structure. Equation \eqref{defchi}  shows that $\rchi^{}_{\scriptstyle{{\mathcal H}'}}$ lies in the tangent space to the orbits of $SL(2,\C)$ on $SL(2,\C)\times \CP^1$ so that $SL(2,\C )$ orbits on the complex 3-dimensional contact manifold $(SL(2,\C)\times \CP^1)/\rchi^{}_{\scriptstyle{{\mathcal H}'}}$ have 2 complex dimensions. Locally there is only one such example (see \cite{PurplePete}) given by the prolongation of the $SL(2,\C )$ action 
\begin{equation}
(z,w)\cdot A =\left(z, \frac{ aw +c}{bw + d}\right), \quad ad-bc=1.
\label{SL2Jaction}
\end{equation}
to the 1-jets of holomorphic functions from $\CP^1$ to $\CP^1$ denoted by to $J^1(\CP^1 , \CP^1)$. The  prolongation of this $SL(2,\C)$ action to $J^2(\CP^1, \CP^1)$ (the 2-jets) is then given in standard local jet-coordinates $(z,w,w_z,w_{zz})$ by,
\begin{equation}
\rho\left(\, (z,w,w_z,w_{zz}), \, A\right)= \left(\, z ,\ \frac{aw+c}{bw+d}, \frac{1}{(bw+d)^2} w_z ,\  \frac{w_{zz}}{(bw+d)^2}- \frac{2bw_z^2}{(bw+d)^3}\,  \right)
\label{jetA}
\end{equation}
This action is not effective due to the center of $SL(2,\C)$ acting trivially. However $PSL(2,\C)$ acts freely and effectively on the subset of 2 jets of immersions which we denote by $J_0^2(\CP^1 , \CP^1)\subset J^2(\CP^1 , \CP^1)$.  These observations lead to the the following theorem.

\begin{thm}\label{EquivT} The Pfaffian system ${\mathcal H}$  is  $SL(2,\C)$ equivariantly equivalent to the jet space of holomorphic immersions, $(J_0^2(\CP^1 , \CP^1), {\mathcal C})$  where $SL(2,\C)$ acts on $J_0^2(\CP^1 , \CP^1) $ via prolongation of the action in equation \eqref{SL2Jaction} and given locally in equation \eqref{jetA}. 
In particular, the map  $\Psi: SL(2,\C) \times \CP^1 \to J_0^2(\CP ^1 , \CP^1)$ given by
\begin{equation}
\Psi(A,[p,q]) = \rho\left(\, ([p,q]A^{-1}, [p,q] A^{-1},1,0) , \, A\, \right)
\label{defPsi}
\end{equation}
is an $SL(2,\C)$ equivariant local diffeomorphism, and $\Phi^* {\mathcal C} = {\mathcal H}$.
\end{thm}

The map $\Psi: SL(2,\C) \times \CP^1 \to J_0^2(\CP ^1 , \CP^1)$  in equation \eqref{defPsi} is given in the local jet coordinates $(z,w,w_z,w_{zz})$  by,
\begin{equation}
\Psi(A, u ) = ( z=  \frac{c-u d}{a-ub}, w= u  ,  w_z=(a-ub)^2 ,w_{zz}= -2 b (a-ub)^3)\quad ad-bc=1,
\label{defLocPsi}
\end{equation}
where  $u=pq^{-1}$ is a local coordinate on the domain $\CP^1$ of $\Psi$ and  $(z,w)$ are local coordinates on the image $\CP^1 \times \CP^1$ of $\Psi$.


\begin{proof} The map $\Psi$ in equation \eqref{defPsi}  is checked to be equivariant by,
$$
\begin{aligned}
\Psi( AA', [p,q]A') &=&&\!\! \!\! \rho\left(\, ([p,q]A'(AA')^{-1}, [p,q]A'(AA')^{-1}, 1,0) , \,  AA' \right) \\
&=&&\! \!\!\! \rho\left(\, \rho\left(\, ([p,q]A, [p,q]A^{-1}, 1,0) ,\, A\, \right),\,  A'\, \right)   \\
&=&&\!\! \!\! \rho \left( \, \Psi(A, [p,q]),\, A' \right) 
\end{aligned}
\quad A, A' \in SL(2,\C).
$$
While it is easy to check that $\Psi$ is two-to-one (due to the center of $SL(2,\C)$), and is therefore an immersion and hence a local diffeomorphism. To finish the proof we now check $\Phi^* {\mathcal C} = {\mathcal H} $.

Using the local coordinates in equation \eqref{defLocPsi}, and   \eqref{defHp} we have
$$
\Psi^* (dw-w_z dz) =  2u \alpha^1 -u^2 \alpha^2 +\alpha^3 \in H'.
$$
While from equations \eqref{defLocPsi} and \eqref{HU} we have,
$$
\Psi^*( dw_z -w_{zz} dz) = 2(a-ub)^2 (\alpha^1-u \alpha^2) \in H .
$$
These coordinates are defined on an open dense subset for the domain of $\Psi$, and therefore $\Psi^* {\mathcal C} = {\mathcal H}$. 
\end{proof}


The de-prolongation of part 2 in Example 4 of \cite{FelsIvey}  (or see \cite{FelsIveyArchive}) shows  that the Monge Amp\'ere form of the Liouville equation given in equation \eqref{LiouEq} is the local quotient of the real version of the contact system ${\mathcal C}$ on $J^2_0(\CP^1, \CP^1 )$ by $SU(2)$. That is,
\begin{equation}\label{defMliou}
{\mathcal M}_{Liou} =\sysE_2/\rhoSU=\sysE_2 /SU(2)
\end{equation}
where $\rhoSU$ the action of $SU(2)$ on $J^2_0(\CP^1, \CP^1 )$ given by restricting \eqref{SL2Jaction} to the subgroup $SU(2)$,  $\sysE_2$ is the real version of ${\mathcal C}$, and ${\mathcal M}_{Liou}$ is the Monge Amp\'ere system for the Liouville equation given by
\begin{equation}
{\mathcal M}_{Liou} =\{ dU-U_z dz - U_{\bar z} d{\bar z},\ dU_z\wedge dz + d U_{\bar z}  \wedge d{\bar z}  , \ {\bf i}(
dU_z\wedge d z - (d{ U_{\bar z}}     -  {\rm e}^U d{z}) \wedge d{\bar z}) \}_\alg
\label{defMLiou}
\end{equation}
in coordinates  $(z,U,U_z)$ where $z,U_z$ are complex, and $U$ is real. Integral manifolds of ${\mathcal M}_{Liou} $ in equation \eqref{defMLiou} in the form $U=U(z,\bar z)$ safisfy the Liouville equation.
$$
2U_{z \bar z} +   {\rm e}^U =0 \ .
$$
See Remark 4.12 in Section 4.2 of \cite{AFCR} for the commutativity of de-prolongation and quotients.

The second part of example 4 from \cite{FelsIvey} demonstrates equation \eqref{defMliou} explicitly by working on the open set of $J^2_0(\CP^1 , \CP^1)$ with homogeneous coordinates $[p,q],[r,s]$ on the domain and target $\CP^1$ respectively  and inhomogeneous local coordinates $z=pq^{-1}$, $w=rs^{-1}$. The quotient map ${\bq}_{\rhoSU}:J^2_0(\CP^1, \CP^1 ) \to J^2_0(\CP^1, \CP^1 )/\rhoSU$ is then expressed in local coordinates by
\begin{equation}
{\bq}_{\rhoSU}=\left(z=z, \ U= 2\log \frac{2 |w_z|}{1+|w|^2}, \ U_z=\frac{w_{zz}}{w_z}-2\frac{\overline{w} w_z}{1+|w|^2} \right).
\label{Lsln}
\end{equation}
We can verify $ {\mathcal M}_{Liou} = \sysE_2/{\bq}_{\rhoSU}$ in equation \eqref{defMliou} by checking ${\bq}_{\rhoSU}^* {\mathcal M}_{Liou} \subset {\mathcal C }$ in local coordinates using equations \eqref{Lsln} and \eqref{defMLiou}.

Finally combining example 4 from \cite{FelsIvey} as summarized above,   Theorem \ref{TMAreduction}, and Theorem \ref{EquivT} we have the commutative diagram as described in equation \eqref{CD0},
\begin{equation} 
\begin{gathered}
\tikzstyle{line} = [draw, -latex', thick]
\begin{tikzpicture}
\node(SLCP) {$\left( SL(2,\C) \times \CP^1 , {\sysE}_1 \right)$};
\node[right of=SLCP, node distance=60mm](B) {$\left( J^2_0(\CP^1 , \CP^1)  , \sysE_2 \right)$};
\node[below of=SLCP, node distance=20mm](I1) {$\left( \HT  \times \CP^1 , {\mathcal M}_{CMC}\right)$};
\node[right of=I1, node distance=60mm](I2) {$\left( J^1(\reals^2, \reals ) , {\mathcal M}_{Liou}\right)$};
\path[line](SLCP) -- node[above] {$\Psi $} (B);
\path[line](B) -- node[left] {$\bq_{\rhoSU}$} (I2);
\path[line](I1) -- node[below] {$\psi$} (I2);
\path[line](SLCP) -- node[left] {$\Pi_{\Delta_{SU(2)}}$} (I1);
\end{tikzpicture}
\end{gathered}
\label{CD1T1}
\end{equation}
where $\Psi$ and $\psi$ are equivalences and $\Psi$ is $SL(2, \C)$ equivariant. The systems satisfy
${\mathcal M}_{CMC}= \sysE_1/\bq_{\Delta_{SU(2)}}$ where
$\sysE_1$ is the real version of ${\mathcal H}$ 
and  ${\mathcal M}_{Liou}= \sysE_2/\bq_{\rhoSU}$ 
where $\sysE_2$ is the real version of the contact system ${\mathcal C}$.

Finally, in order to prove Theorem \ref{LCMC} which maps solutions of Liouville's equation to CMC-1 hypersurfaces, we
let    $\Phi:J_0^2(\CP ^1 , \CP^1) \to SL(2,\C) \times \CP ^1 $  be the (right) inverse of $\Psi$ which by equation \eqref{defLocPsi} is given in coordinates by,
\begin{equation}
\Phi( z, w, w_z, w_{zz}) = \left( \frac{1}{2 w_z^{3/2}}\left[ \begin{matrix} 
 2 w_z^2-w w_{zz} & -w_{zz},\\
  z w w_{zz} +2 w w_z  -2 z w_z^2 &   (z w_{zz} + 2w_z) \end{matrix} \right] , \ 
u = w   \right).
\label{bigPhi}
\end{equation}
We use equation \eqref{bigPhi} to show the following.

\begin{corollary}\label{PT1} Let $U(z,\bar z)$ be a solution to the Liouville equation, then $ \tilde \phi:\C \to \HT $ 
\begin{equation}
\begin{aligned}
\tilde \phi\left(z,U \, \right) &=&&\sfj^{-1}\left( \frac{1}{2} \left[  \begin{matrix} {\mathrm e}^{\frac{U}{2}}+|U_z|^2 {\mathrm e}^{-\frac{U}{2}} &
-\bar z{\mathrm e}^{\frac{U}{2}} -   ( \bar z U_{\bar z}  +2)U_z {\mathrm e}^{-\frac{U}{2}}
\\
- z {\mathrm e}^{\frac{U}{2}} -  (z{U_z}  +2)U_{\bar z }{\mathrm e}^{-\frac{U}{2}} 
&   |z|^2 {\mathrm e}^{\frac{U}{2}}+  |zU_z  +2|^2  {\mathrm e}^{-\frac{U}{2}} \end{matrix} \right]  \ \right)
\end{aligned}
\label{LtoCMC}
\end{equation}
where $\sfj$ is defined in equation \eqref{SL2H3} is a CMC-1 surface in $\HT $.
\end{corollary}

\begin{proof} 
Let $\phi= \psi^{-1}: J^1(\R^2,\R)\to \HT  \times \CP^1$ from equation \eqref{CD1T1}. By the equivalence of the systems $\sysMCMC$ and ${\mathcal M}_{Liou}$ the projection in $\HT$ of $\phi$ given by $\phi_1=\pi_1 \circ \phi$ evaluated on a solution to the Liouville equation is a CMC-1 surface. The corollary will then follow by writing $\phi_1$ in coordinates and evaluating on a solution to the Liouville equation.

We choose the local  cross-section $\Sigma: J^2_0(\CP^1,\CP^1)/\rhoSU \to J^2_0(\CP^1,\CP^1)$ to the $SU(2)$ action in \eqref{jetA} (restricted to $SU(2)$)
\begin{equation*}
\Sigma:=\left( z=z,\ w=0, w_z=\frac{1}{2} {\rm e}^{\frac{U}{2}},\ w_{zz}=\frac{U_z}{2}{\rm e}^{\frac{U}{2}} \ \right)
\end{equation*}
 so that $\Pi_{\rhoSU} \circ \Sigma ={\rm Id}$ with $\Pi_{\rhoSU}$ in equation \eqref{Lsln}.
We then have from the diagram in equation \eqref{CD1},  $\phi = \psi^{-1} = \Pi_{\Delta_{SU(2)}} \circ \Psi^{-1} \circ \Sigma= \Pi_{\Delta_{SU(2)}} \circ \Phi \circ \Sigma$ where $\Pi_{\Delta_{SU(2)}}$ is given in equation \eqref{PiAp}. 

By identifying $J^2_0(\CP^1,\CP^1)/\rhoSU=J^1(\reals^2,\reals)$ as done in equation \eqref{Lsln} with local coordinates $(z,U,U_z)$, we let $\hat \phi =  \pi_1\circ \Phi \circ \Sigma :J^1(\reals^2, \reals) \to SL(2,\C) $ which with $\Phi$ in equation \eqref{bigPhi} gives

\begin{equation}
\hat \phi_1(z,\bar z, U,U_z, U_{\bar z})  = \frac{\sqrt{2}}{2}
\left[\begin{array}{cc}
{\mathrm e}^{\frac{U}{4}} & -\mathit{U_z} \,{\mathrm e}^{-\frac{U}{4}}
\\
 -z {\mathrm e}^{\frac{U}{4}}  &\left(z\mathit{U_z}  +2\right)  {\mathrm e}^{-\frac{U}{4}} 
\end{array}\right]
\label{IMCMC}
\end{equation}

Then the map  $\tilde \phi = \pi_1\circ \Pi_{\Delta_{SU(2)}} \circ \Phi \circ \Sigma =\Pi_{SU(2)}\circ \hat \phi$ 
where $\Pi_{SU(2)} $  from equation \eqref{defpisu2} gives $\tilde \phi: J^1(\R^2,\R)\to  \HT$  as
\begin{equation}
\tilde \phi = \sfj^{-1}\left(\hat \phi_1 \cdot (\hat \phi_1)^\dagger\right).
\label{tilphi}
\end{equation}
Substituting equation \eqref{IMCMC} into equation \eqref{tilphi} gives the right hand side of equation \eqref{LtoCMC}. This implies that if $U(z,\bar z)$ is a solution to the Liouville equation \eqref{LiouEq} (and hence an integral manifold of ${\mathcal M}_{Liou}$) then \eqref{LtoCMC} is a CMC-1 surface.  
 \end{proof}



Using equations \eqref{LtoCMC}, \eqref{SL2H3}  and \eqref{defj} to solve for $x^i$ yields,
\begin{equation}
\begin{aligned}
x^0 & = &&\frac{1}{4} \left(  (1+|z|^2 ) {\rm e}^{\frac{U}{2}}  +(|U_z|^2+|zU_z+2|^2    {\rm e}^{-\frac{U}{2}} \right) \\
x^1&=&& -\frac{1}{4} \left( (z+ \bar z  ) {\rm e}^{\frac{U}{2}}+\left(U_z U_{\bar z}(z+\bar z) +2 ( U_z+U_{\bar z})  \right){\rm e}^{-\frac{U}{2}} \right)
\\
x^2&=&& 
-\frac{{\bf i}}{4}\left( (z-\bar z  ) {\rm e}^{\frac{U}{2}}+
\left(  U_z U_{\bar z}(z-\bar z) + 2 ( U_{\bar z}-U_z)  \right){\rm e}^{-\frac{U}{2}} \right)
\\
x^3 &=&&\frac{1}{4} \left(  (1-|z|^2 ) {\rm e}^{\frac{U}{2}}  +(|U_z|^2-|zU_z+2|^2    {\rm e}^{-\frac{U}{2}} \right) .
\end{aligned}
\label{RIM}
\end{equation}
Expressing these equations in terms of $U=U(x,y)$ leads to Theorem \ref{LCMC} in the Introduction. It is also worth noting that substitution from  equation \eqref{Liousol} for $U$ in terms of $f(z)$ into equation \eqref{RIM} also provides a Weierstass formula for CMC-1 surface in  $\HT$ in terms of the holomorphic function $f(z)$, see also Theorem 1.1 in \cite{LeviCMC}.

%% file: appendix.tex
\appendix
\section{Quotient Space  Lemmas for Section 4}

\newcommand{\cdotn}{\!\cdot\!} 
The quotient representation for ${\sysICMC}$ utilizes the following technical results.

\begin{lem} \label{PreLpsimap}  Let $(A,[p,q]) \in SL(2,\C ) \times \CP^1$ and let $R\in SU(2)$ be a matrix such that $[p,q] R =[0,1]$. 
Let $\Delta_{SU(2)}$ denote the diagonal right action from \eqref{SL2action}
restricted to $SU(2)$
Then map $\gamma :(SL(2,\C ) \times \CP^1)/{\Delta_{SU(2)}} \to \HT\times \CP^1$ defined pointwise by
\begin{equation}
\gamma\left( (A ,[p,q])\cdotn SU(2) \, \right)  = \Pi_{S^1}( AR), 
\label{defgam}
\end{equation}
where $\Pi_{S^1}$ is given by equation \eqref{defpis1}, is a well-defined diffeomorphism.
\end{lem}

\begin{proof} In order to prove this map is well-defined we first show that the function $\alpha:SL(2,\C)\times\CP^1 \to SL(2,\C)/S^1$ defined by
\begin{equation}
\alpha\left( (A ,[p,q]) \right) = \Pi_{S^1}(AR).
\label{defalpha}
\end{equation}
is independent of the choice of $R$. Suppose $R,R'\in SU(2)$ with 
$[p,q]R=[p,q]R' =[0,1]$.   Let $S = R^{-1}R' \in SU(2)$. 
Then $[0,1]=[p,q]R'=[p,q] R S=[0,1]S$ implies that $S\in SU(2)\cap B = S^1$. Therefore $\pi_{S^1}(AR')=\pi_{S^1}(ARS)=\pi_{S^1}(AR)$ since $S\in S^1$. 

We next show that $\alpha$ is invariant under the diagonal $SU(2)$ action
by an element $U\in SU(2)$.  First, note that if $[p',q'] = [p,q] U$ and 
$[p,q] R = [0,1]$ then $[p',q'] U^{-1} R = [0,1]$.
Hence 
$$
\alpha( ( A,[p,q]) \cdotn U ) = \Pi_{S^1}( A U(U^{-1} R)) = \Pi_{S^1}(AR) =\alpha( ( A,[p,q]))
$$
and $\alpha$ is invariant under $\Delta_{SU(2)}$. 
Thus there exists a unique 
$\gamma: (SL(2,\C) \times \CP^1)/{\Delta_{SU(2)}} \to SL(2,\C) /S^1$ satisfying $\gamma\circ\bq_{\Delta_{SU(2)}}= \alpha$, which must be given by \eqref{defgam}.  

Since $\alpha( (A, [0,1])) = \Pi_{S^1}(A)$ for any $A \in SL(2,C)$, $\alpha$ is surjective and hence $\gamma$ is surjective.
It remains to check that $\gamma$ is injective. 
Let $(A,[p,q]), (A',[p',q']) \in SL(2,\C) \times \CP^1$.  Supposing that
\begin{equation}
\gamma(\,(A,[p,q])\cdotn SU(2)\,) = \gamma(\,(A',[p',q'])\cdotn SU(2)\,),
\label{showin}
\end{equation}
we will show that there exists $P\in SU(2)$ such that $(A',[p',q'])= (A,[p,q])\cdotn P$. 

Since $SU(2)$ acts transitively on $\CP^1$, there exists $U\in SU(2)$ such that $[p',q']=[p,q] U$. As before, if $[p,q]R=[0,1]$ then 
\begin{equation}
[p',q'] U^{-1} R = [p,q] R= [0,1] .
\label{derRp}
\end{equation}
Equations \eqref{defalpha} and \eqref{showin} then imply that
$\Pi_{S^1}(A' U^{-1} R) = \Pi_{S^1}(A R)$. 
Thus there exists $T_\theta\in S^1$ such that $A' U^{-1} R = A R T_\theta$
and so
\begin{equation}\label{Aprimeformula}
A' = A R T_\theta R^{-1} U
\end{equation}
Letting $P=R T_\theta R^{-1} U\in SU(2)$, 
from \eqref{derRp} it follows that
\begin{equation}\label{Ponpq} 
[p,q] P = [p,q] R T_\theta R^{-1} U 
= [0,1] T_\theta R^{-1} U = [0,1] R^{-1} U = [p',q'].
\end{equation}
From \eqref{Aprimeformula} and \eqref{Ponpq} we obtain $(A',[p',q'])= (A,[p,q])\cdotn P$, 
so the map $\gamma$ is injective.

To see that $\gamma$ is smooth, note that for $[p,q] \in \CP^1$ the matrix
\begin{equation}\label{defRpq}
R_{p,q} = \dfrac{1}{\sqrt{ |p|^2 + |q|^2}}\begin{pmatrix} q &\overline{p} \\ -p &\overline{q} \end{pmatrix}\in SU(2)
\end{equation}
satisfies $[p,q] R_{p,q} = [0,1]$. Then the map $\widehat \alpha: SL(2,\C) \times (\C^2\backslash{(0,0)} ) \to \HT \times \CP^1$ defined by $\widehat\alpha( A, (p,q)) = \Pi_{S^1}\left( A R_{p,q}\right)$ is clearly smooth, but also invariant under the $\Cm$-action
of scaling $(p,q)$.  Since it induces the map $\alpha$ in \eqref{defalpha}, the latter map is smooth, and hence $\gamma$ is smooth.
\end{proof}

Because $\gamma$ in Lemma \ref{PreLpsimap} is a diffeomophism, it follows that 
$\Pi_{\Delta_{SU(2)}}: SL(2,\C)\times \CP^1\to \HT \times \CP^1$ defined by
\begin{equation}
\Pi_{\Delta_{SU(2)}}( A,[p,q]) =\gamma ( (A,[p,q]) \cdot SU(2)) .
\label{defPi}
\end{equation}
is a representation of the quotient map $\bq_{\Delta_{SU(2)}}$,
where again $\Delta_{SU(2)}$ denotes the right diagonal action from \eqref{SL2action} restricted to $SU(2)$.

\begin{lem}\label{Lpsimap} 
Let $\iota([p,q])=[q,-p]$. 
Then
\begin{equation}
\Pi_{\Delta_{SU(2)}}( A,[p,q]) = \left(\ \Pi_{SU(2)}(A), \iota( [p,q]A^{-1})\ \right).
\label{PiAp}
\end{equation}
\end{lem}

\begin{proof} Let $(A,[p,q]) \in SL(2,\C)\times \CP^1$ and let $R_{p,q}$
be as in \eqref{defRpq}.  Then by equations \eqref{defpis1}, \eqref{defgam} and \eqref{defPi} we have
$$
\begin{aligned}
\Pi_{\Delta_{SU(2)}}(A,[p,q]) &=&& \gamma( (A,[p,q])\cdotn SU(2)) \\
&=&& \Pi_{S^1}(A R_{p,q})\\
&=&&(\ \Pi_{SU(2)}(A R_{p,q}) ,\ \Pi_\Borel(A R_{p,q})\ ) \\
&=&& (\ \Pi_{SU(2)}(A) ,\ \Pi_\Borel(A R_{p,q})\ ).
\end{aligned}
$$
Thus we only need to check the second component in equation \eqref{PiAp}.

For the right action of the subgroup $\Borel$ on $SL(2,\C)$  and  $\Pi_\Borel(A) = [a,c]\in \CP^1$ from equation \eqref{defPiB}. The second component in the left side of equation \eqref{PiAp}
 is then computed to be
 $$
\Pi_\Borel(A R_{p,q}) = \Pi_\Borel\left(   \frac{1}{\sqrt{|p|^2+|q|^2}} 
\left[ \begin{matrix} a & b \\ c & d \end{matrix} \right]
 \left[ \begin{matrix} q & \bar p \\ -p & \bar q \end{matrix} \right] \right)=
 \frac{1}{\sqrt{|p|^2+|q|^2}} [ aq-bp, cq-dp] 
 $$
while the second term in the right side of equation \eqref{PiAp} is
$$
\iota([p,q] A^{-1})=\iota( [dp-cq, -bp+aq]) = [aq-bp, cq-dp].
$$
Equation \eqref{PiAp} is now verified.
\end{proof}

\section{The EDS $\sysH$}\label{A2}


\newcommand{\Gr}{\mathrm{Gr}}
\newcommand{\pir}{\mbox{\raisebox{.3ex}{$\pi$}}} 
\newcommand{\Cnon}{\C^3_0}  
\newcommand{\vxi}{\mathbf{\xi}}  
\newcommand{\Np}{\widecheck{N}} 
\newcommand{\sysC}{\mathcal C}   
\newcommand{\sysK}{\mathcal K}

Here we derive the system $\sysH$ from \eqref{Hdef} whose integral surfaces
are equivalent to the holomorphic null curves in $SL(2,\C)$ used in \cite{BryantCMC} to obtain CMC-1 immersions in $\HT$.  (Recall that such curves are related to lifts of integral surfaces for $\sysM$ by modification using the $SU(2)$-valued solution of 
an equation of Lie type.)  

We begin by defining trivializations of the holomorphic
tangent bundle $T_{1,0}SL(2,\C)$ and its associated Grassmann bundle $\Gr_1 (SL(2,\C))$ of 1-dimensional complex subspaces.  Let 
$${\bq}_{\Cm} : T_{1,0}SL(2,\C)\backslash \{ {\bf 0}\} \to \Gr_1 (SL(2,\C)),$$ 
where $\bf 0$ denotes the zero section, be the fiber-wise projectivization map.
The left-invariant vector fields $Z_i$ on $SL(2,\C)$, given in local coordinates in \eqref{defZ}, allow us to define a trivialization $\tau:T_{1,0} SL(2,\C) \to SL(2,\C) \times \C^3$ such that
\begin{equation*}
\tau(  \xi^1Z_1+\xi^2 Z_2 + Z_3 \xi^3 \vert_A) = (A; ( \xi^1,\xi^2,\xi^3 ) ), 
\qquad A \in SL(2,\C)\ .
\end{equation*}
This in turn induces a trivialization $\mu:\Gr_1 (SL(2,\C))  \to  SL(2,\C) \times \CP^2$ such that 
$$
 \mu \circ \bq_{\Cm}(  \xi^1Z_1+\xi^2 Z_2 + Z_3 \xi^3 \vert_A) = (A ; [\xi^1,\xi^2,\xi^3]),  
 $$
where $[\xi^1,\xi^2,\xi^3]$ denote homogeneous coordinates on $\CP^2$.

Let $\Cnon=\C^3\backslash\{(0,0,0)\}$. From here on we use $\tau$ to identify systems and group actions on  $T_{1,0}SL(2,\C)\backslash \{ {\bf 0}\}$ with those on $SL(2,\C)\times \Cnon$, and similarly use $\mu$ to identify
objects on $\Gr_1(SL(2,\C))$ with those on $SL(2,\C) \times \CP^2$.
For example, the action of $SL(2,\C)$ on $T_{1,0}SL(2,\C)$ by push-forward under
right-multiplication induces an action on $SL(2,\C)\times \C^3$ for which 
$\tau$ is equivariant.  Letting $\Lambda(C)$ for $C\in SL(2,\C)$ denote the matrix representing
the $\C$-linear extension of the adjoint representation of $C^{-1}$, that is of $\operatorname{Ad}_{C^{-1}} $, on $T_{1,0}SL(2,\C)$ with respect to the basis $Z_i$ at the identity, the latter action is 
$$
(A;\vxi) \cdot C  = (A C; \vxi \Lambda(C)^t   )\ , \quad \vxi = (\xi^1, \xi^2, \xi^3).
$$
Each action commutes with projectivization, so there are induced (right) actions of $SL(2,\C)$ on $\Gr_1(SL(2,\C))$ and $SL(2,\C) \times \CP^2$ for which $\mu$ is equivariant.

Let $\widetilde N \subset T_{1,0}SL(2,\C)$ be the 5-dimensional complex manifold defined by  
$$
\widetilde N=\{ Z\in T_{1,0}SL(2,\C)  \mid \kappa(Z,Z) = 0,  Z\ne 0\},
$$ 
where $\kappa$ is the Killing form, extended $\C$-linearly and by left-translation to each tangent space on $SL(2,\C)$.   We will in fact work more with its image under
the trivialization,
\begin{equation*}
N = \mu(\widetilde N) = \{ (A; \vxi) 
\mid 
(\xi^1)^2+\xi^2 \xi^3 = 0,  \vxi \neq 0 \}
\subset SL(2,\C) \times \Cnon.
\end{equation*}
It is easy to verify that the holomorphic null curves of \cite{BryantCMC} are exactly 
those regular immersions $s: \C \to SL(2,\C)$ satisfying $s_* T_{(1,0)}\C \subset \widetilde N$
(e.g., see the definition before Theorem A in \cite{BryantCMC}).
We will call these {\defem null curves in $SL(2,\C)$}.

As $N$ is invariant under the scaling action $(A, \vxi) \cdot a = (A, a \vxi)$
for $a\in \Cm$, it has a well-defined projection $\Np = \pir_{\Cm}(N) \subset SL(2,\C) \times \CP^2$ where $\pir_{\Cm}: SL(2,\C) \times \C^3 \to SL(2,\C) \times \CP^2$. This is a 4-dimensional complex manifold which is $SL(2,\C)$- invariant (with respect to the right diagonal action in equation \eqref{SL2action}).  We use the Segre embedding to define a holomorphic diffeomorphism
$\Psi: SL(2,\C) \times \CP^1 \to \Np$ given by
\begin{equation}
\Psi(A; [p,q]) =  (A; [pq  , q^2, -p^2] ) 
\label{deftoR}
\end{equation}
where $[p,q]$ are homogeneous coordinates on $\CP^1$.  This satisfies the equivariance condition
\begin{equation*}
\Psi(A C ; [p,q] C ) = \pir_{\Cm} (A C; (pq  , q^2, -p^2) \Lambda(C)^t  ), \quad C\in SL(2,\C).
\end{equation*}
The action on the left hand side is the diagonal free right action of $SL(2,\C )$ on $SL(2,\C) \times \CP^1$ given in equation \eqref{SL2action}.

Let $\tilde \sysC \subset T^*_{(1,0)} \Gr_1(SL(2,\C))$ be the canonical contact system\footnote{At a point $E \in \Gr_1(SL(2,\C))$ this system is spanned by the pullbacks from $SL(2,\C)$ of $(1,0)$-forms annihilating $E$.}, which is a holomorphic Pfaffian system of rank 2, 
and let $\tilde \sysK =\bq_{\Cm}^* \tilde \sysC$ be the rank 2 Pfaffian system on $T_{1,0}SL(2,\C)$.
The integral curves of the latter system which are transverse to the projection $\pi:T_{1,0}SL(2,\C) \to SL(2,\C)$
 are lifts by differentiation of immersed holomorphic curves into $SL(2,\C)$; in other words, given a holomorphic $SL(2,\C)$-valued function$A(z)\in SL(2,\C) $ for $z\in \C$, then
$dA/dz$ gives an integral curve of $\tilde \sysK$.
Let $\sysC$ and $\sysK$ be the corresponding rank 2 systems on $SL(2,\C) \times \CP^2$ 
and $SL(2,\C) \times \Cnon$ under trivializations $\mu$ and $\tau$ respectively.
Then with coordinates $( \xi^1, \xi^2, \xi^3)$ on $\Cnon$, 
\begin{equation}
\sysK = \{ \xi^1 \alpha^2 - \xi ^2 \alpha^1,\ \xi^1\alpha^3 - \xi^3 \alpha^1,\ \xi^2 \alpha^3  - \xi^3 \alpha^2\ \}_\diff.
\label{EDSK}
\end{equation}
(Of the three generators listed, exactly two are linearly independent at each point.)

Integral curves of $\sysK$ lying in $N$ which are transverse to $\pi:N \to SL(2,\C)$ correspond under $\tau$ to lifts 
of null curves in $SL(2,\C)$. 
Using the diffeomorphism $\Psi:SL(2,\C) \times \CP^1 \to \Np$,
we define $\sysH=\Psi^* \sysC$, a holomorphic rank 2 Pfaffian system on $SL(2,\C) \times \CP^1$.  
Using equations \eqref{deftoR} and \eqref{EDSK} we compute that
\begin{equation}
\sysH = \{  \ p q \alpha^2 - q^2 \alpha^1, \  pq \alpha^3 +p^2 \alpha^1,\  q^2 \alpha^3 +p^2 \alpha^2\ \}_\diff
\label{Idef}
\end{equation}
where again $[p,q]$ are homogeneous coordinates on $\CP^1$.
Since $\sysC$, $\sysK$ and $N$ are $SL(2,\C)$-invariant and $\Psi$ is equivariant,
then $\sysH$ is invariant under the diagonal right action defined in \eqref{SL2action}.  

We now summarize the above discussion for use in Section 4 as follows.

\begin{thm} \label{defsysH}   Let   $\sysH$ be the holomorphic rank 2 Pfaffian system on $SL(2,\C) \times \CP^1$ given by
\begin{equation*}
\sysH = \{  \ p q \alpha^2 - q^2 \alpha^1, \  pq \alpha^3 +p^2 \alpha^1,\  q^2 \alpha^3 +p^2 \alpha^2\ \} _\diff
\end{equation*}
which is $SL(2,\C)$ invariant under the right diagonal action in equation \eqref{SL2action} where $[p,q]$ are homogeneous coordinates on $\CP^1$. If  $\, s:\C \to SL(2,\C)\times \CP^1$ is a holomorphic integral curve of $\sysH$ which is transverse to $\pi: SL(2,\C) \times \CP^1 \to SL(2,\C)$, then $\pi \circ s:\C \to SL(2,
\C)$ is a holomorphic null curve. Conversely every holomorphic null curve $s: \C \to SL(2,\C)$ gives rise to an integral manifold of $\sysH$ transverse to $\pi$ given by $\Psi^{-1}\circ \pir_{\Cm}\circ \tau (\frac{ds}{dz})$. 
\end{thm}

\section{The Vessiot coframe}\label{A3}
\newcommand{\ditto}{\ldots}

In this section we will show that the Vessiot algebra of the Darboux-integrable system $\sysJ$ is $\fsu(2)$.
We use the methods of \S4 in \cite{FelsIvey}, and in order to do so, we will represent the
quotient  of $\sysJ$  in equation \eqref{MAonSL2J} on $SL(2,\C)\times \C$ by the action of the characteristic symmetry given by the $S^1$-action in equation \eqref{S14A} by
restricting the system $\sysI$ to the  cross-section $SL(2,\C)\times \R \subset SL(2,\C)\times \C$. We will also restrict our attention to an open set of $SL(2,\C)$ where we can use coordinates.  Accordingly, let
$$\scrU = \{ (A,\np) \in SL(2,\C) \times \C \mid a \ne 0, \np \in \R^+ \}.$$
On the 7-dimensional manifold $\scrU$ we use as coordinates $a,b,c$ and their conjugates as well as the
positive real coordinate $\np$.

Let $\iota$ denote the inclusion of $\scrU$ into $SL(2,C) \times \C$, and let
$\sysJU = \iota^* \sysJ$ be the restriction.
By computing the values of the left-invariant forms $\alpha^1, \alpha^2,\alpha^3$ in our coordinates we obtain from equation \eqref{MAonSL2J}
$$\sysJU=\{ \Re(d \rd{a}-b\rd{c}), 
\Re\left( 2(c\rd{a}-a\rd{c}) - \np\left(\dfrac{\rd{b}}{a} +\dfrac{b}{a}(b \rd{c}-d \rd{a}) + \bar{a}\rd{\bar{c}}-\bar{c}\rd{\bar{a}}\right)\right), 
\Im\left({\ditto} \right) \}_\alg,$$
where $d=(bc+1)/a$ and $\ditto$ indicates the same 1-form as in the previous expression. In addition we pull back the singular bundle $\widehat R$ and from \eqref{defhatRinf} compute that
$$
\iota^* \prRinf = \spanC \{ c\rd{a}-a\rd{c}, a\rd{\np}-4\np\rd{a}\},
$$
showing that $c/a$ and $\np/a^4$ are Darboux invariants of $\sysJU$.  It is usually advantageous to
include the invariants as coordinates; in order to solve for $a,c$ in terms of the invariants, let
$\np=k^4$ for $k\in \R^+$, and introduce Darboux invariant coordinates
$$w=c/a, \qquad q=k/a.$$

The Vessiot algebra is determined by constructing an adapted complex coframe on $\scrU$; the coframe is not unique, and possibly only local, but the real Lie algebra determined by the structure equations of the coframe is unique up to isomorphism.  To begin the construction, we first select an adapted coframe using a basis of generators for $\sysJ \otimes \C$ (in fact, just using the above generators expressed in terms of the new coordinates) along with the differentials of the Darboux invariants and their conjugates:
\begin{multline*}
( \theta^1_0=\rd\log(k/q)-bk/q \rd{w} +\rd\log(k/\qbar) -\bbar k/\qbar \rd{\wbar}, \\
\theta^2_0 = k(b\rd{q}-q\rd{b})-bq\rd{k}+(b^2 k^2 -2q^{-2})\rd{w} - k^4 \qbar^{-2}\rd{\wbar},
\theta^3_0=\overline{\theta^2_0}, \pi^1=\rd{w}, \pi^2=\rd{q}, \overline{\pi^1}, \overline{\pi^2}).
\end{multline*}
In Step 1, iterated brackets of are taken of the duals of $\pi^1,\pi^2$ to obtain a new frame satisfying structure equations with fewer nonzero coefficients.   The corresponding coframe has
$$
\theta^1_1 = \tfrac14 q^4 k b \theta^1_0 + \tfrac18 q^3 \theta^2_0 + (\tfrac14 q^3 k^{-4} - \tfrac18 q^5 k^{-2} b^2)\theta^3_0, 
\theta^2_1 = \tfrac18 q^5( -\theta^1_0 + q b k^{-3}\theta^3_0),
\theta^3_1 = -\tfrac1{16}q^7 k^{-4} \theta^3_0,
$$
with $\pi^1, \pi^2, \overline{\pi^1}, \overline{\pi^2}$ unchanged from now on.
In Step 2 we choose an `origin' $o$, in this case defined by $k=2^{1/4}$, 
$q=2$ and $b=w=0$, and make a constant-coefficient change of coframe to ensure
that the $\theta$'s are pure imaginary at $o$:
$$\theta^1_2 = \ri \theta^1_1, \quad \theta^2_2=\ri \theta^2_1, \quad \theta^3_2 = \theta^1_1 + \tfrac12 \theta^3_1.$$
Because the complex span of the $\theta$'s is closed under conjugation, there is a matrix $P$ of functions such that $\theta^i = P^i_j \overline{\theta^j}$ where $1\le i,j \le 3$.  In Step 3, we obtain a matrix $K$ by setting the coordinates
$b,\overline{b}, \overline{q}, \overline{w}$ and $k$ in the entries of $P$ equal to their values at the origin (but with $K$ still depending on $q$ and $w$), and letting $\theta^i_3 = K^i_j \theta^j_2$ .  The resulting coframe satisfies the structure equations
\begin{equation*}
\begin{aligned} 
d\theta^1 &= \tfrac12 \theta^2 \wedge \theta^3  + \dfrac{2\ri}q \pi^2 \wedge \theta^2  + \dfrac{8\ri}{q^3} \pi^1 \wedge \pi^2 -\dfrac{8\ri (b^2 k^2 q^2-2)}{k^4 \qbar^3} \pibars{1} \wedge \pibars{2},\\
d\theta^2 &= \theta^3 \wedge \theta^1 + \dfrac{4\ri}{q^2} \pi^1 \wedge \theta^3
+ \dfrac{32\ri b q}{k^3 \qbar^3} \pibars{1} \wedge \pibars{2}, \\
d\theta^3 &= \tfrac12 \theta^1 \wedge \theta^2 -\dfrac{2\ri}{q^2} \pi^1 \wedge \theta^2 -\dfrac{2\ri}{q} \pi^2 \wedge \theta^1 + \dfrac{8}{q^3} \pi^1 \wedge\pi^2 -8\dfrac{b^2 k^2 q^2 +2}{\qbar^3 k^4}\pibars{1} \wedge \pibars{2},
\end{aligned}
\end{equation*}
where for convenience we have dropped the subscript 3 on the $\theta$'s.
The real constants $C^i_{jk}$ such that $d\theta^i \equiv \tfrac12 C^i_{jk} \theta^j \wedge \theta^k$ modulo $\pi^a$, $\pibars{a}$ are the structure constants of the Vessiot algebra, which in this case is readily identifiable as $\fsu(2)$.
